\documentclass[reqno]{amsart}
\usepackage{amsmath,amssymb,amsfonts,caption}
\usepackage{graphicx}
\usepackage{euscript}
\usepackage{enumerate}
\usepackage{verbatim}
\usepackage{epsfig}
\usepackage{float}

\usepackage[normalem]{ulem}
\usepackage{color}
\usepackage[usenames,dvipsnames,svgnames,table]{xcolor}
\definecolor{orange}{rgb}{1,0.5,0}
\newtheorem{theorem}{Theorem}

\newtheorem{lemma}{Lemma}
\newtheorem{proposition}{Proposition}

\renewcommand{\epsilon}{\varepsilon}
\renewcommand{\phi}{\varphi}
\renewcommand{\le}{\leqslant}
\renewcommand{\ge}{\geqslant}
\newcommand{\eps}{\varepsilon}
\newcommand{\lbd}{\lambda}
\newcommand{\ds}{\, ds}
\newcommand{\dr}{\, dr}

\newcommand{\cR}{\mathcal R}

\DeclareMathOperator{\e}{e}
\DeclareMathOperator{\Exp}{exp}

\usepackage{dsfont}
   \newcommand{\N}{\ensuremath{\mathds N}}
   \newcommand{\R}{\ensuremath{\mathds R}}

\begin{document}
\title[]
   {An eco-epidemiological model with general functional response of predator to prey}
\author{Lopo F. de Jesus}
\address{L. F. de Jesus\\
 	Centro de Matemática e Aplicações (CMA-UBI), Universidade da Beira Interior\\
	and Instituto Superior de Ciências da Educação\\
	{Rua Sarmento Rodrigues, Lubango, Angola}
   }
   \email{lopo.jesus@ubi.pt}
\author{C\'esar M. Silva}
\address{C. M. Silva\\
      Centro de Matemática e Aplicações (CMA-UBI), Universidade da Beira Interior\\
      6201-001 Covilh\~a, Portugal}
\email{csilva@ubi.pt}
\author{Helder Vilarinho}
\address{H. Vilarinho\\
    Centro de Matemática e Aplicações (CMA-UBI), Universidade da Beira Interior\\
      6201-001 Covilh\~a, Portugal}
\email{helder@ubi.pt}
\date{\today}
\thanks{L. de Jesus, C. M. Silva and H. Vilarinho were partially supported by FCT through CMA-UBI (project UIDB/00212/2020). L. F. de Jesus was also supported by INAGBE.}
\begin{abstract}
We consider a nonautonomous eco-epidemiological model with general functions for predation on infected and uninfected preys as well as general functions associated to the vital dynamics of the susceptible prey and predator populations. We obtain persistence and extinction results for the infected prey based on assumptions on systems related to the dynamics in the absence of infected preys. We apply our results to eco-epidemiological models constructed from predator-prey models existent in the literature. Some illustrative simulation is undertaken.
\end{abstract}
\maketitle
\section{Introduction}
The description of the dynamics of eco-epidemiological systems is a subject that have been receiving increasing attention by the researchers interested in mathematical biology. The inclusion of infected classes in predator-prey models has shown that eco-epidemiological dynamics can show several differences to the original models. In particular, the inclusion of a disease in the preys or in the predators have impact on the population size of the predator-prey community~\cite{Hethcote-Wang-Han-Ma-TPB-2004, Koopmans-Wilbrink-Conyn-Natrop-Nat-Vennema-Lancet-2004}.

Additionally, to make models more realistic, it is important, in many situations, to consider time varying parameters. For instance, it is well known in epidemiology that incidence rates are seldom subject to periodic seasonal fluctuations. In the context of eco-epidemiological models, several periodic systems have been studied in the literature~\cite{Dobson-QRV-1988, Friend-H-2002, Koopmans-Wilbrink-Conyn-Natrop-Nat-Vennema-Lancet-2004, Krebs-Blackwell-Scientific-Publishers-1978, Niu-Zhang-Teng-AMM-2011}. In~\cite{Niu-Zhang-Teng-AMM-2011} a class of general non-autonomous eco-epidemiological models with disease in the prey, containing the periodic case as a very particular situation, is considered and threshold conditions for the extinction and persistence of the infected preys are obtained. Related to the periodic version of this model, in~\cite{Silva-JMAA-2017}, it is proved the existence of an endemic periodic orbit. In~\cite{Niu-Zhang-Teng-AMM-2011, Silva-JMAA-2017}, it is assumed that only infected preys are predated. More recently, based on this model,~\cite{Lu-Wang-Liu-DCDS-B-2018} proposed a family of models that include predation on uninfected preys described by a bilinear functional response and obtained threshold conditions for the extinction and persistence of the disease.

In the previous papers, the functional response of the predator to prey is given by some particular function. Also the vital dynamics of predator and prey is assumed to follow some particular law. In this paper we generalize the models in~\cite{Lu-Wang-Liu-DCDS-B-2018,Niu-Zhang-Teng-AMM-2011} by considering general functions corresponding to the predation on infected and uninfected prey and also to the vital dynamics of uninfected prey and predator populations. Namely, we consider the following eco-epidemiological model:
\begin{equation}\label{eq:principal}
\begin{cases}
S'=G(t,S)-a(t)f(S,I,P)P-\beta(t)SI\\
I'=\beta(t)SI-\eta(t)g(S,I,P)I-c(t)I\\
P'=h(t,P)P+\gamma(t)a(t)f(S,I,P)P+\theta(t)\eta(t)g(S,I,P)I
\end{cases},
\end{equation}
where $S$, $I$ and $P$ correspond, respectively, to the susceptible prey, infected prey and predator, $\beta(t)$ is the incidence rate of the disease, $\eta(t)$ is the predation rate of infected prey, $c(t)$ is the death rate in the infective class, $\gamma(t)$ is the
rate converting susceptible prey into predator (biomass transfer), $\theta(t)$ is the rate of converting infected prey into predator, $G(t,S)$ and $h(t,P)P$ represent the vital dynamics of the susceptible prey and predator populations, respectively, {$a(t)f(S,I,P)$} is the predation of susceptible prey and $\eta(t)g(S,I,P)$ represent the predation of infected prey. It is assumed that only susceptible preys $S$ are capable of reproducing, i.e, the infected prey is removed by death (including natural and disease-related death) or by predation before having the possibility of reproducing.

The objective of this work is to discuss the uniform strong persistence and extinction of the infectives $I$ in system~\eqref{eq:principal}. Recall that the infectives are \emph{uniformly strong persistent} in system~\eqref{eq:principal} if there exist $0<m_1<m_2$ such that for every solution $(S(t),I(t)P(t))$ of~\eqref{eq:principal} with positive initial conditions $S(t_0),I(t_0),P(t_0)>0$, we have
\[
m_1 < \liminf_{t \to \infty} I(t) \le \limsup_{t \to \infty} I(t) < m_2,
\]
and we say that the infectives $I$ go to \emph{extinction} in system~\eqref{eq:principal} if
\[
\displaystyle \lim_{t \to +\infty} I(t) = 0,
\]
for all solutions of~\eqref{eq:principal} with positive initial conditions.
For biological reasons we will only consider for system~\eqref{eq:principal} solutions with initial conditions in the set $(\R_0^+)^3$.

Our approach is very different to the one in~\cite{Niu-Zhang-Teng-AMM-2011} and~\cite{Lu-Wang-Liu-DCDS-B-2018}. In fact, we want to discuss the extinction and strong persistence of the infectives in system~\eqref{eq:principal}, having as departure point some prescribed behaviour of subsystems related to the dynamics of preys and predators in the absence of disease. We will assume that we have global asymptotic stability of solutions of these special {bi-dimensional subsystems}  (see condition S\ref{cond-11}) in Section~\ref{section:asymptotically-stable-in-prey/predator subspace}). Thus, to apply our results to specific situations in the literature, one must first verify that the underlying refered subsystems satisfy our assumptions or, looking at our results differently, we can construct an eco-epidemiological model from a previously studied predator-prey model (the uninfected subsystem) that satisfies our assumptions. We believe that this approach is interesting since it highlights the relation of the dynamics of the eco-epidemiological model with the behaviour of the predator-prey model used in its construction.

We note that, similarly to the thresholds obtained in~\cite{Lu-Wang-Liu-DCDS-B-2018}, our thresholds for extinction and uniform strong persistence are not sharp.
In spite of this, unlike the conditions for extinction and strong persistence in~\cite{Lu-Wang-Liu-DCDS-B-2018}, that rely on parameters that can not, in principle, be computed explicitly (note that conditions (22) and (43) in~\cite{Lu-Wang-Liu-DCDS-B-2018} depend on $q_1$), our thresholds can be directly obtained from the parameters and the limit behavior of the predator-(uninfected) prey subsystem.

To illustrate our findings, in Section~\ref{examples} some predator-prey models available in the literature, satisfying our assumptions, are considered and thresholds conditions for the corresponding eco-epidemiological model automatically obtained from our results: in our Example 1, we consider the situation where $f\equiv 0$ in system~\eqref{eq:principal}, corresponding to a generalized version of the situation studied in~\cite{Niu-Zhang-Teng-AMM-2011}; in Example 2, we obtain a particular form for the threshold conditions in the context of periodic models and particularize our result for a model constructed from the predator-prey model in~\cite{Goh-AN-1976}; in Example 3, we consider a model with Michaelis-Menten (or Holling-type I) functional response of predator to infected prey and a Holling-type II functional response of predator to susceptible prey; finally, in Example 4 we consider the eco-epidemiological model obtained from an uninfected subsystem with ratio-dependent functional response of predator to prey, a type interaction considered as an attempt to overcome some know biological paradoxes observed in models with Gause-type interaction and again obtain the corresponding results for the eco-epidemiological model, based on the discussion of ratio-dependent predator-prey systems in~\cite{Hsu-Hwang-Kuang-MB-2001}. For all these examples we present some simulation that illustrate our conclusions.

\section{Eco-epidemiological models with asymptotically stable behavior in the predator-uninfected prey subspace}\label{section:asymptotically-stable-in-prey/predator subspace}

We will assume the following hypothesis concerning the parameter functions and the functions $f$, $g$, $G$ and $h$ appearing in our model~\eqref{eq:principal}:
\begin{enumerate}[S$1$)]
\item \label{cond-1} The real valued functions $a$, $\beta$, $\eta$, $c$, $\gamma$ and $\theta$ are bounded, nonnegative and continuous;
\item \label{cond-3} The real valued functions $f$, $g$, $G$ and $H(t,x)=h(t,x)x$ are locally Lipschitz and functions {$f$ and $g$} are nonnegative. For fixed
    $x,z \ge 0$, functions $y\mapsto f(x,y,z)$ and $y \mapsto g(x,y,z)$ are nonincreasing; for fixed $y,z \ge 0$, function $x \mapsto g(x,y,z)$ is nonincreasing; for fixed $x,y
    \ge 0$, {function $z \mapsto f(x,y,z)$ is nonincreasing and} function $z \mapsto g(x,y,z)$ is nondecreasing;
\end{enumerate}

Our next assumption relates to the $\omega$-limit of solutions of~\eqref{eq:principal} and is usually fulfilled by mathematical models in eco-epidemiology.
\begin{enumerate}[S$1$)]
\setcounter{enumi}{2}
\item \label{cond-12} Each solution of~\eqref{eq:principal} with positive initial condition is bounded and there is a bounded region $\mathcal R$ that contains the $\omega$-limit of all solutions of~\eqref{eq:principal} with positive initial conditions.
\end{enumerate}
Notice in particular that condition~S\ref{cond-12}) implies that there is $L>0$ such that
\[
\limsup_{t \to +\infty} \, (S(t)+I(t)+P(t)) < L,
\]
for all solutions $(S(t),I(t),P(t))$ of~\eqref{eq:principal} with positive initial conditions.\\

To proceed, we need to consider two auxiliary equations and one auxiliary system. First, we consider the equation
\begin{equation}\label{eq:auxiliary-S(t)}
s'=G(t,s),
\end{equation}
that corresponds to the dynamics of uninfected preys in the absence of infected preys and predators (the first equation in system~\eqref{eq:principal} with $I = 0$ and $P = 0$). We assume the following properties for the solutions of~\eqref{eq:auxiliary-S(t)}:\\

\begin{enumerate}[S$1$)]
\setcounter{enumi}{3}
\item \label{cond-9} Each solution $s(t)$ of~\eqref{eq:auxiliary-S(t)} with positive initial condition is bounded, bounded away from zero, and globally attractive on $]0,+\infty[$, that is $|s(t)-v(t)| \to 0$ as $t \to +\infty$ for each solution $v(t)$ of~\eqref{eq:auxiliary-S(t)} with positive initial condition.
\end{enumerate}

\,

The second auxiliary equation we consider is the equation
\begin{equation}\label{eq:auxiliary-P(t)}
y'=h(t,y)y,
\end{equation}
that corresponds to the dynamics of predators in the absence of the considered preys (the third equation in system~\eqref{eq:principal} with $I=0$ and $S=0$).
We need the following property for the solutions of~\eqref{eq:auxiliary-P(t)}:\\

\begin{enumerate}[S$1$)]
\setcounter{enumi}{4}
\item \label{cond-10} Each fixed solution $y(t)$ of~\eqref{eq:auxiliary-P(t)} with positive initial condition is bounded and globally attractive on $[0,+\infty)$.
\end{enumerate}

\,

Finally, starting from the \emph{uninfected subsystem}, that is, the system that describes the behavior of preys and predators in the absence of infected preys (the first and third equations of system~\eqref{eq:principal} with $I = 0$), given by
\begin{equation}\label{eq:uninfected-system}
\begin{cases}
x'=G(t,x)-a(t)f(x,0,z)z\\
z'=h(t,z)z+\gamma(t)a(t)f(x,0,z)z
\end{cases},
\end{equation}
we assume that we are able to construct families of {auxiliary subsystems}:
\begin{equation}\label{eq:auxiliary-system-SP-1}
\begin{cases}
{x'=G_{1,\eps}(t,x)-a(t)f(x,0,0)\hat{z}_\epsilon(t)-v(\eps) \rho(t) x}\\
z'=h_{1,\eps}(t,z)z+\gamma(t)a(t)f(x,v(\eps) \rho^u,z)z
\end{cases}
\end{equation}
where  $(\hat x_\epsilon(t),\hat z_\epsilon(t))$ is a solution of
\begin{equation}\label{eq:auxiliary-system-SP-2}
\begin{cases}
{x'=G_{2,\eps}(t,x)}\\
z'=h_{2,\eps}(t,z)z+\gamma(t)a(t)f(x,0,z)z+v(\eps) \rho(t)g(x,{0},z)
\end{cases}
\end{equation}
{satisfying} the following assumptions.

\,

\begin{enumerate}[S$1$)]
\setcounter{enumi}{5}
\item \label{cond-11} The following holds for systems~\eqref{eq:auxiliary-system-SP-1} and~\eqref{eq:auxiliary-system-SP-2}:
    \begin{enumerate}
    \item[S6.1)]\label{cond-61} for sufficiently small $\eps>0$, the functions $G_{i,\eps}$ and $h_{i,\eps}$, $i=1,2$, are continuous, the functionals
    $\eps \mapsto G_{i,\eps}$ and $\eps \mapsto h_{i,\eps}$, $i=1,2$, are continuous, $G_{1,0}=G_{2,0}=G$, $h_{1,0}=h_{2,0}=h$,
    \[
    G_{1,\eps}(t,x) \le G(t,x) \le G_{2,\eps}(t,x)
    \]
    and
    \[
    \quad h_{1,\eps}(t,x) \le h(t,x) \le h_{2,\eps}(t,x);
    \]
    \item[S6.2)] the real valued function $v:[0,+\infty[ \to \R$ verifies $v(\eps)>0$ for $\eps \in \, ]0,+\infty]$, $v(0)=0$ and is differentiable near $\eps=0$ with
        \[
        A<v'(\eps)<B,
        \]
    for some $A,B>0$ and sufficiently small $\eps \ge 0$;
    \item[S6.3)]\label{cond-63} there are constants $\rho^u,\rho^\ell$ such that, for all $t\ge 0$,
        \[
        0<\rho^{\ell} \le \rho(t) \le \rho^u;
        \]
    \item[S6.4)] \label{eq:subsyst-1} there is a family of nonnegative solutions, $\{(x^*_{1,\eps}(t), \, z^*_{1,\eps}(t))\}$ of system~\eqref{eq:auxiliary-system-SP-1}, one
        solution for each $\eps\ge0$ sufficiently small, {depending on a solution $(x_{2,\epsilon}^*(t),z_{2,\epsilon}^*(t))$ of system~\eqref{eq:auxiliary-system-SP-2},} such that each solution in the family is globally asymptotically stable in a set containing the set $(\R^+)^2$ and the function
        \[
        \eps \mapsto (x^*_{1,\eps}(t), \, z^*_{1,\eps}(t)) \quad \text{is continuous;}
        \]
    \item[S6.5)]\label{cond-65} 	{the} family of nonnegative solutions $\{(x^*_{2,\eps}(t), \, z^*_{2,\eps}(t))\}$ of system~\eqref{eq:auxiliary-system-SP-2}, one solution for each
        $\eps\ge0$ sufficiently small, {is} such that each solution in the family is globally asymptotically stable in a set containing the set $(\R^+)^2$ and the function
        \[
        \eps \mapsto (x^*_{2,\eps}(t), \, z^*_{2,\eps}(t)) \quad \text{is continuous.}
        \]
    \end{enumerate}
\end{enumerate}
{We write $x^*_{1,0}=x_1^*$, $x^*_{2,0}=x^*_{2}$,  $z^*_{1,0}=z^*_{1}$ and $z^*_{2,0}=z^*_{2}$} for the components of the solutions in {S6.4) and S6.5)} corresponding to $\eps=0$. {For the continuity}  {of the functionals in S6.1), {{S6.4) and S6.5)}} we consider the usual supremum norm, $\|\cdot\|_0$ (notice that, by
S\ref{cond-12}) the solutions are bounded). Note that we only aim to control two suitable families of perturbations of the uninfected subsystem, so that
condition~S\ref{cond-11}) is sufficiently flexible to adapt to a wide range of uninfected subsystems associated to the eco-epidemiological models.


We emphasize that our setting includes several of the most common functional responses for both functions $f$ and $g$. Writing \[f(S,I,P)=\frac{kS^\alpha}{h(S,I,P)}\,\, \text{ and }\,\, g(S,I,P)=\frac{kP^\alpha}{h(S,I,P)}\]  we may consider the following cases:
\begin{enumerate}
\item[-] Holling-type I: $h(S,I,P)=1$, $\alpha=1$
\item[-]  Holling-type II : $h(S,I,P)={(1+m(S+I))}$, $\alpha=1$
\item[-]  Holling-type III: $h(S,I,P)={(1+m(S+I))}$, $\alpha>0$
\item[-] Holling-type IV: $h(S,I,P)={(a+b(S+I)+c(S+I)^2)}$, $\alpha=1$
\item[-] Beddington-De Angelis: $h(S,I,P)={(a+b(S+I)+cP)}$, $\alpha=1$
\item[-] Crowley-Martin: $h(S,I,P)={(a+b(S+I)+cP+d(S+I)P)}$, $\alpha=1$.
\end{enumerate}

\subsection{Main results}
In this subsection, we will establish our results on the extinction and uniform strong persistence of the infective prey in system~\eqref{eq:principal}, assuming conditions~S\ref{cond-1}) to S\ref{cond-11}).
Given a function $f$ we will use throughout the paper the notations $f^\ell=\inf_{t \ge 0} f(t)$, $f^u=\sup_{t \ge 0} f(t)$ and, for a $\omega$-periodic function $f$ we use the notation $\bar f= (1/\omega) \int_0^\omega f(s)\ds$.

We define
\begin{equation}\label{eq:cR0-ext}
\cR^\ell(\lambda)=\liminf_{t \to +\infty} \int_t^{t+\lambda} \beta(s){x^*_{1}}(s)-\eta(s)g({x^*_{1}}(s),0,{z^*_{2}}(s))-c(s) \ds
\end{equation}
{where we still denote by ${x^*_{1}}(t)$ and ${z^*_{2}}(t)$ the components of solutions in systems~S6.4) and~S6.5), with $\epsilon=0$, and}
\begin{equation}\label{eq:cR0-per}
\cR^u(\lambda)=\limsup_{t \to +\infty} \int_t^{t+\lambda} \beta(s)s^*(s)-\eta(s)g(s^*(s),0,y^*(s))-c(s) \ds.
\end{equation}
where $s^*(t)$ and $y^*(t)$ are particular solutions, respectively, of~\eqref{eq:auxiliary-S(t)} and~\eqref{eq:auxiliary-P(t)} with positive initial conditions. 

As we will see in the following, using the global attractivity of solutions of~\eqref{eq:auxiliary-S(t)} and~\eqref{eq:auxiliary-P(t)} in $]0,+\infty[$ and the global
attractivity of {solutions given at S6.4) and S6.5)} we can easily conclude that~\eqref{eq:cR0-ext} is independent of the particular solutions
~{ considered in S6.4) and S6.5)}.
  Similarly, it is easy to conclude that~\eqref{eq:cR0-per} is independent of the particular solutions of~\eqref{eq:auxiliary-S(t)} and~\eqref{eq:auxiliary-P(t)} with positive
  initial conditions considered.

\begin{proposition}\label{teo:independence-of-R0-of-solutions}
The numbers~\eqref{eq:cR0-ext} and~\eqref{eq:cR0-per} are independent, respectively, of the particular solutions { considered in S6.4) and S6.5)} and of the particular solutions of~\eqref{eq:auxiliary-S(t)} and~\eqref{eq:auxiliary-P(t)}  with positive initial conditions chosen.
\end{proposition}
{\begin{proof}
{ Let $(x_1^*(t),z_1^*(t))$, $(x_2^*(t),z_2^*(t))$ and  
$(\bar x_1^*(t),\bar z_1^*(t))$, $ (\bar x_2^*(t),\bar z_2^*(t))$ be two distinct pairs of nonnegative solutions of~\eqref{eq:auxiliary-system-SP-1} and \eqref{eq:auxiliary-system-SP-2} as in S6.4) and S6.5)}.
Let $\delta>0$. By~ {S\ref{cond-11})}, for $t\geq  {T_\delta}$ sufficiently large, we have
\[
{
x_1^*(t)-\delta \leqslant \bar x_1^*(t) \leqslant x_1^*(t)+\delta  {\quad \text{ and } \quad}
z_2^*(t)-\delta \leqslant \bar z_2^*(t) \leqslant z_2^*(t)+\delta.
}
\]
Additionally,  {by~S\ref{cond-3})} there is $c>0$ such that, for  {every $t\geq T_\delta$},
\small{
\[
\begin{split}
& \left|\int_t^{t+\lambda} \beta(s)x_1^*(s)-\eta(s)g(x_1^*(s),0,{z_2^*}(s))-c(s) \ds-\int_t^{t+\lambda} \beta(s){\bar x_1^*(s)}-\eta(s)g({\bar x_1^*}(s),0,{\bar z_2^*}(s))-c(s) \ds \right|\\
&  {\le \int_t^{t+\lambda} \beta(s)\left|x_1^*(s)-{\bar x_1^*}(s)\right|+\eta(s)\left|g(x_1^*(s),0,{z_2^*}(s))-g({\bar x_1^*}(s),0,{\bar z_2^*}(s))\right| \ds}\\
&  {\le \lambda\beta^u \delta+2\lambda \eta^u\phi(\delta)},
\end{split}
\]
}
with $\phi(\delta)\to 0$ as $\delta \to 0$. We conclude that, for every  $\delta>0$,
 \[
\begin{split}
& \limsup_{t \to +\infty}\int_t^{t+\lambda} \beta(s)x_1^*(s)-\eta(s)g(x_1^*(s),0,{z_2^*}(s))-c(s) \ds-\lambda\beta^u \delta-2\lambda \eta^u\phi(\delta)\\
\leqslant & \limsup_{t \to +\infty}\int_t^{t+\lambda} \beta(s){\bar x_1^*}(s)-\eta(s)g({\bar x_1^*}(s),0,{\bar z_2^*}(s))-c(s) \ds\\
\leqslant & \limsup_{t \to +\infty}\int_t^{t+\lambda} \beta(s)x_1^*(s)-\eta(s)g(x_1^*(s),0,{z_2^*}(s))-c(s) \ds+\lambda\beta^u \delta+2\lambda \eta^u\phi(\delta),
\end{split}
\]

Taking $\liminf$ instead of $\limsup$, and using the same reasoning we conclude that we have similar inequalities. Thus $\cR^\ell(\lambda)$ is independent of the chosen solution.
 {In the same way we can prove that $\cR^u(\lambda)$ is also independent of the particular solutions chosen} and the result follows.
\end{proof}}

\begin{theorem}\label{teo:main-extinction}
Assume that conditions S\ref{cond-1}) to S\ref{cond-10}) hold. Assume further that either $G(t,S)=\Lambda(t)-\mu(t)S$ and  $g(S+I,0,P)\le g(S,I,P)$ or $g$ does not depend on $I$. If there is $\lbd>0$ such that $\cR^u(\lambda)<0$, then the infectives in system~\eqref{eq:principal} go to extinction.
\end{theorem}

\begin{proof}
Assume that there is $\lambda>0$ such that $\cR^u(\lambda)<0$ and let $s^*(t)$ and $y^*(t)$ be particular solutions, respectively, of~\eqref{eq:auxiliary-S(t)} and~\eqref{eq:auxiliary-P(t)} with positive initial conditions. Since functions $\beta$ and $\eta$ are bounded, there are $\kappa>0$, $t_0>0$ and $\eps_0>0$ such that, for  $t \ge t_0$ and $\delta \in \ ]0,\eps_0]$, we have
\begin{equation}\label{eq:hip-extinction}
\int_t^{t+\lambda} \beta(s)(s^*(s)+\delta)-\eta(s)g(s^*(s)+\delta,0,y^*(s)-\delta)-c(s) \ds \le -\kappa < 0.
\end{equation}

Let $(S(t),I(t),P(t))$ be a solution of~\eqref{eq:principal} with positive initial conditions. We will prove first that
\begin{equation}\label{eq:not-strong-persist}
\liminf_{t \to +\infty} I(t) =0.
\end{equation}
Assume that~\eqref{eq:not-strong-persist} does not hold. Then, there is $\eps>0$ such that
$I(t)>\eps$ for all sufficiently large $t$. By the first equation of~\eqref{eq:principal} we have
\[
S' \le G(t,S)
\]
and thus $S(t)\le s(t)$, where $s(t)$ is the solution of~\eqref{eq:auxiliary-S(t)} with $s(t_0)=S(t_0)$. By condition~S\ref{cond-9}), given
$\eps \in \, ]0,\eps_0]$, we have
$S(t) \le s^*(t)+\eps$, for all sufficiently large $t$.

By the third equation of~\eqref{eq:principal}, we have
\[
P' \ge h(t,P)P
\]
and thus $P(t)\ge y(t)$, where $y(t)$ is the solution of~\eqref{eq:auxiliary-P(t)} with $y(t_0)=P(t_0)$. By condition~S\ref{cond-10}), given
$\eps \in \, ]0,\eps_0]$, we have
$P(t) \ge y^*(t)-\eps$, for all sufficiently large $t$.

{When $G(t,S)=\Lambda(t)-\mu(t)S$, 
\[(S+I)'\leq \Lambda(t)-\mu(t)S-c(t)I\leq \Lambda(t)-\mu(t)(S+I),\]
and consequently, for sufficiently large $t$
\[S(t)+I(t)\leq s^*(t)+\epsilon.\]
Under this assumption on $G$, by the second equation of~\eqref{eq:principal}, since we assumed that $g(S+I,0,P)\le g(S,I,P)$, we have  
\begin{align}
I'
& \le  \left[\beta(t)(s^*(t)+\eps)-\eta(t)g(s^*(t)+\eps,0,y^*(t)-\eps)-c(t)\right]I, \label{eq:g does not dep on I}
\end{align}
for all sufficiently large $t$. 
Notice that, for a general $G$, if $g$ does not depend on $I$ we have $g(S,I,P)\geq g(s^*(t)+\epsilon,0,y^*(t)-\epsilon)$ and we still obtain inequality~\eqref{eq:g does not dep on I}.
}

Denoting by $\lfloor \alpha \rfloor$ the integer part of $\alpha$ and integrating the previous equation, we get
\[
\begin{split}
I(t)
& \le I(t_0)\Exp\left\{\int_{t_0}^t \beta(r) (s^*(r)+\eps)-\eta(r)g(s^*(r)+\eps,0,y^*(r)-\eps)c(r)\dr\right\}\\
& \le I(t_0)\e^{\lambda (\beta^u (s^*)^u+\eps\beta^u)}\\
& \quad  \times\Exp\left\{\int_{t_0}^{t_0+\lfloor \frac{t-t_0}{\lambda} \rfloor \lambda} \beta(r) (s^*(r)+\eps)-\eta(r)g(s^*(r)+\eps,0,y^*(r)-\eps)-c(r)\dr \right\}\\
& \le I(t_0)\e^{-\lfloor (t-t_0)/\lambda \rfloor \kappa} \e^{\lambda (\beta^u (s^*)^u+\eps\beta^u)},
\end{split}
\]
for all sufficiently large $t$. Since $\lfloor (t-t_0)/\lambda \rfloor \kappa \to +\infty$ as $t\to +\infty$,
we get a contradiction to the hypothesis that there is $\eps>0$ such that $I(t)>\eps$ for sufficiently large $t$. We conclude that~\eqref{eq:not-strong-persist} holds.

Let $\eps>0$. Next we will prove that for sufficiently large $t$
\begin{equation}\label{eq:not-weak-persist}
I(t) \le \eps \e^{h\lbd},
\end{equation}
where
$$h=\sup_{t \ge 0} \left| \beta(t) (s^*(t)+\eps_0)-\eta(t)g(s^*(t)+\eps_0,0,y^*(t)-\eps_0)-c(t) \right|.
$$
By~\eqref{eq:not-strong-persist}, there exists $t_1\ge t_0$ such that $I(t_1)<\eps$.

Assume, by contradiction that~\eqref{eq:not-weak-persist} does not hold. Then, there is $t_2>t_1$ such that
$I(t_2) > \eps \e^{h\lbd}$. Since $I(t_1)<\eps$, there is $t_3\in \ ]t_1,t_2[$ such that $I(t_3)=\eps$ and
$I(t)>\eps$, for all $t \in \ ]t_3,t_2[$. Integrating we get, by~\eqref{eq:hip-extinction},
\[
\begin{split}
\eps\e^{h\lbd}
& < I(t_2)\\
& \le I(t_3)\exp\left\{\int_{t_3}^{t_2} \beta(r) (s^*(r)+\eps)-\eta(r)g(s^*(r)+\eps,0,y^*(r)-\eps)-c(r)\dr\right\}\\
& \le \eps \exp\left\{\int_{t_3+\lfloor{(t_2-t_3)}/\lbd\rfloor \lbd}^{t_2} \beta(r) (s^*(r)+\eps_0)-\eta(r)g(s^*(r)+\eps_0,0,y^*(r)-\eps_0)-c(r)\dr\right\}\\
& \le \eps\e^{h\lbd},
\end{split}
\]
witch is a contradiction. Thus, we conclude that~\eqref{eq:not-weak-persist} holds and, since $\eps \in \ ]0,\eps_0]$ is arbitrary, we conclude that
$I(t) \to 0$ as $t \to 0$, as claimed.
\end{proof}

\begin{theorem}\label{teo:main-persistence}
Assume that conditions S\ref{cond-1}) to S\ref{cond-12}) and  S\ref{cond-11}) hold. If there is $\lbd>0$ such that $\cR^\ell(\lambda)>0$, then the infectives in system~\eqref{eq:principal} are uniformly strong persistent.
\end{theorem}

\begin{proof}
Assume that there is $\lambda>0$ such that $\cR^\ell(\lambda)>0$ and let us fix particular families of solutions of systems~\eqref{eq:auxiliary-system-SP-1} and
\eqref{eq:auxiliary-system-SP-2}, respectively $(x^*_{1,\eps}(t),z^*_{1,\eps}(t))$ and $(x^*_{2,\eps}(t),z^*_{2,\eps}(t))$, with positive initial conditions and satisfying {N}6.4)
and {N}6.5). Then, we can choose $t_0>0$ and $\eps_0>0$ such that, for $t \ge t_0$ and $\delta \in \, [0,\eps_0]$ we have
\begin{equation}\label{eq:cond-R0>0}
\int_t^{t+\lambda} \beta(s)({x^*_{1}}(s)-\delta)-\eta(s)g({x^*_{1}}(s)-\delta,\delta,{z^*_{2}}(s)+\delta)-c(s) \ds \ge \kappa > 0.
\end{equation}

Let $(S(t),I(t),P(t))$ be a solution of~\eqref{eq:principal} with positive initial conditions.
We will prove first that there is $\eps>0$ such that
\begin{equation}\label{eq:weak-persist}
\limsup_{t \to +\infty} I(t)\ge \frac{v(\eps){\rho^\ell}}{(1+\beta^u)(1+\theta^u\eta^u)}>0.
\end{equation}
Assume that for all sufficiently small $\eps>0$
\[
\limsup_{t \to +\infty} I(t)<\frac{v(\eps){\rho^\ell}}{(1+\beta^u)(1+\theta^u\eta^u)}.
\]
Then, we conclude that there is $t_1>t_0$, such that
\begin{equation}\label{eq:I<vrho}
I(t)<\frac{v(\eps){\rho^\ell}}{(1+\beta^u)(1+\theta^u\eta^u)}<v(\eps)\rho(t),
\end{equation}
for each $t \ge t_1$. By the first and third equations of~\eqref{eq:principal} and the inequalities in S6.1) we have
\[
\begin{cases}
{S' \le G_{2,\eps}(t,S)}\\
P'\le h_{2,\eps}(t,P)P+\gamma(t)a(t)f(S,0,P)P+v(\eps)\rho(t) \theta(t)\eta(t) g({S},{0},P)
\end{cases}.
\]
{Let} $(\hat x_\eps(t), \hat z_\eps(t))$ be the solution of
\[
\begin{cases}
{x' = G_{2,\eps}(t,x) }\\
z' = h_{2,\eps}(t,z)z+\gamma(t)a(t)f(x,0,z)z+v(\eps)\rho(t) \theta(t)\eta(t)g(x,{0},z)
\end{cases}
\]
with $\hat x_\eps(t_1)=S(t_1)$ and $\hat z_\eps(t_1)=P(t_1)$. We have $S(t)\le \hat x_\eps(t)$ and $P(t)\le \hat z_\eps(t)$ for $t \ge t_1$.
By the global stability assumption in~{N}6.5), we have
\[
\left|x^*_{2,\eps}(t)-\hat x_\eps(t)\right|\to 0 \quad \text{and} \quad \left|z^*_{2,\eps}(t)-\hat z_\eps(t)\right| \to 0, \quad \text{as} \ t \to +\infty
\]
and, by continuity, again according to~{N}6.5), we have for sufficiently large $t$
\[
\begin{split}
\left|{x^*_{2}(t)}-\hat x_\eps(t)\right|
& \le\left|x^*_{2}(t)-x^*_{2,\eps}(t)\right|+\left|x^*_{2,\eps}(t)-\hat x_\eps(t)\right|\\
& \le\|x^*_{2}-x^*_{2,\eps}\|_0+\left|x^*_{2,\eps}(t)-\hat x_\eps(t)\right|\\
& \le \phi_1(\eps),
\end{split}
\]
and
\[
\begin{split}
\left|{z^*_{2}}(t)-\hat z_\eps(t)\right|
& \le\left|{z^*_{2}}(t)-z^*_{2,\eps}(t)\right|+\left|z^*_{2,\eps}(t)-\hat z_\eps(t)\right|\\
& \le\|{z^*_{2}}-z^*_{2,\eps}\|_0+\left|z^*_{2,\eps}(t)-\hat z_\eps(t)\right|\\
& \le \phi_2(\eps),
\end{split}
\]
with $\phi_1(\eps),\phi_2(\eps)\to 0$ as $\eps \to 0$. In particular, for sufficiently large $t$,
\begin{equation}\label{eq:bound-sist-persist-P}
S(t) \le \hat s_\eps(t) \le \phi_1(\eps)+{x^*_{2}}(t) \quad \text{ and } \quad P(t) \le \hat z_\eps(t) \le \phi_2(\eps)+{z^*_{2}}(t).
\end{equation}

On the other hand, {by \eqref{eq:I<vrho} and} the first and third equations of~\eqref{eq:principal}, we have
{
\[
\begin{cases}
S' \ge G_{1,\eps}(t,S)-a(t){f(S,0,0)\hat z_\eps(t)}-v(\eps)\rho(t) S\\
P' \ge h_{1,\eps}(t,P)P+\gamma(t)a(t)f(S,v(\eps)\rho^u,P)P
\end{cases}
\]
}
Letting $(\tilde x_\eps(t), \tilde z_\eps(t))$ be the solution of
{
\[
\begin{cases}
x' = G_{1,\eps}(t,x)-a(t){f(S,0,0)\hat z_\eps(t)}-v(\eps)\rho(t) x\\
z' = h_{1,\eps}(t,z)z+\gamma(t)a(t)f(x,v(\eps)\rho^u,z)z
\end{cases}
\]
}
with $\tilde x(t_1)=S(t_1)$ and $\tilde z(t_1)=P(t_1)$, we have $S(t) \ge \tilde x_\eps(t)$ and $P(t)\ge \tilde z_\eps(t)$, for all $t\ge t_1$.
By the global stability assumption in~{N}6.4), we have
\[
\left|x^*_{1,\eps}(t)-\tilde x_\eps(t)\right|\to 0 \quad \text{and} \quad \left|z^*_{1,\eps}(t)-\tilde z_\eps(t)\right| \to 0, \quad \text{as} \ t \to +\infty.
\]
and, by the continuity property in~{S6.4}), for sufficiently large $t$, we have
\[
\begin{split}
\left|{x^*_{1}}(t)-\tilde x_\eps(t)\right|
& \le\left|{x^*_{1}}(t)-x^*_{1,\eps}(t)\right|
+\left|x^*_{1,\eps}(t)-\tilde x_\eps(t)\right|\\
& \le\|{x^*_{1}}-x^*_{1,\eps}\|_0
+\left|x^*_{1,\eps}(t)-\tilde x_\eps(t)\right|\\
& \le \psi_1(\eps),
\end{split}
\]
and
\[
\begin{split}
\left|{z^*_{1}}(t)-\tilde z_\eps(t)\right|
& \le\left|{z^*_{1}}(t)-z^*_{1,\eps}(t)\right|
+\left|z^*_{1,\eps}(t)-\tilde z_\eps(t)\right|\\
& \le\|{z^*_{1}}-z^*_{1,\eps}\|_0
+\left|z^*_{1,\eps}(t)-\tilde z_\eps(t)\right|\\
& \le \psi_2(\eps),
\end{split}
\]
with $\psi(\eps)\to 0$ as $\eps \to 0$. In particular, for sufficiently large $t$,
\begin{equation}\label{eq:bound-sist-persist-S}
S(t) \ge \tilde x_\eps(t) \ge {x^*_{1}}(t)-\psi_1(\eps) \quad \text{and} \quad P(t) \ge \tilde z_\eps(t) \ge {z^*_{1}}(t)-\psi_2(\eps).
\end{equation}

By the second equation in~\eqref{eq:principal},~\eqref{eq:cond-R0>0},~\eqref{eq:bound-sist-persist-P} and~\eqref{eq:bound-sist-persist-S} we get, for $t\ge t_1$,
\small{
\[
\begin{split}
& \int_t^{t+\lambda} \beta(s)S(s)-\eta(s)g(S(s),I(s),P(s))-c(s) \ds\\
& \ge \int_t^{t+\lambda} \beta(s)({x^*_{1}}(s)-\psi_1(\eps))-\eta(s)g({x^*_{1}}(s)-\psi_1(\eps),v(\eps)\rho^u,{z^*_{2}}(s)+\phi_2(\eps))-c(s) \ds
\ge \kappa.
\end{split}
\]
}
Thus, choosing $\eps>0$ such that {$\max\{\phi_2(\eps),\psi_1(\eps),v(\eps)\rho^u\}<\eps_0$}, we have
{
\[
\begin{split}
I(t) &
= I(t_1) \Exp\left\{\int_{t_1}^t \beta(s)S(s)-\eta(s)g(S(s),I(s),P(s))-c(s) \ds\right\}\\
& \ge I(t_1) \Exp\left\{\int_{t_1}^t \beta(s)({x^*_{1}}(s)-\psi_1(\eps))ds\right\}\\
&\quad{\times\Exp}\left\{\int_{t_1}^t-\eta(s)g({x^*_{1}}(s)-\psi_1(\eps),v(\eps)\rho^u,{z^*_{2}}(s)+\phi_2(\eps))-c(s) \ds\right\}\\
& \ge I(t_1)\e^{-\lbd(\beta^u\psi_1(\eps)+\eta^u g(({x^*_{1}})^u-\psi_1(\eps),v(\eps)\rho^u,({z^*_{2}})^\ell+\phi_2(\eps){)}+c^u)}\\
& \quad 	{\times\Exp}\left\{\int_{t_1}^{t_1+\lfloor (t-t_1)/\lbd \rfloor \lbd} \beta(s)({x^*_{1}}(s)-\psi_1(\eps))ds\right\}\\
&\quad {\times\Exp}\left\{\int_{t_1}^{t_1+\lfloor (t-t_1)/\lbd \rfloor \lbd} -\eta(s)g({x^*_{1}}(s)-\psi_1(\eps),v(\eps)\rho^u,{z^*_{2}}(s)+\phi_2(\eps))-c(s)
\ds\right\}\\
& \ge I(t_1) \e^{\lfloor (t-t_1)/\lbd \rfloor \kappa } \e^{-\lbd(\beta^u\psi_1(\eps)+\eta^u g(({x^*_{1}})^u{-\psi_1(\eps)},v(\eps)\rho^u,({x^*_{1}})^\ell+\phi_2(\eps))+c^u)},
\end{split}
\]
}
a contradiction to the fact that, according to~{S\ref{cond-12})}, $I(t)$ is bounded. We conclude that~\eqref{eq:weak-persist} holds.

Next we will prove that there is $m_1>0$ such that for any solution $(S(t),I(t),P(t))$ with positive initial condition,
\begin{equation}\label{eq:strong-persist}
\liminf_{t \to +\infty} I(t) > m_1.
\end{equation}

Assume that~\eqref{eq:strong-persist} does not hold. Then, given $\eps \in ]0,\eps_0[$, there exists a sequence of initial values $(x_n)_{n \in \N}$, with $x_n=(S_n,I_n,P_n)$ and
$S_n>0$, $I_n>0$ and $P_n>0$ such that
\begin{equation*}
\liminf_{t \to +\infty} I(t,x_n) < \frac{\rho^u v(\eps/n^2)}{(1+\theta^u\eta^u)(1+\beta^u)},
\end{equation*}
where $(S(t,x_n),I(t,x_n),P(t,x_n))$ denotes the solution of~\eqref{eq:principal} with initial conditions $S(0)=S_n$, $I(0)=I_n$, and $P(0)=P_n$.
By~\eqref{eq:weak-persist}, given $n \in \N$, there are two
sequences $(t_{n,k})_{k \in \N}$ and $(s_{n,k})_{k \in \N}$ with
\[
s_{n,1} < t_{n,1} < s_{n,2} < t_{n,2} < \cdots < s_{n,k} < t_{n,k} < \cdots
\]
and $\displaystyle \lim_{k \to +\infty} s_{n,k}  = +\infty$, such that
\begin{equation}\label{eq:maj-eps-eps1a}
I(s_{n,k},x_n) = \frac{\rho^\ell v(\eps/n)}{(1+\theta^u\eta^u)(1+\beta^u)}, \quad I(t_{n,k},x_n) = \frac{\rho^u v(\eps/n^2)}{(1+\theta^u\eta^u)(1+\beta^u)}
\end{equation}
and, for all $t \in ]s_{n,k},t_{n,k}[$,
\begin{equation}\label{eq:maj-eps-eps2}
\frac{\rho^u v(\eps/n^2)}{(1+\theta^u\eta^u)(1+\beta^u)} < I(t,x_n) <
\frac{\rho^\ell v(\eps/n)}{(1+\theta^u\eta^u)(1+\beta^u)}.
\end{equation}

By the second equation in~\eqref{eq:principal} and {S\ref{cond-12})}, for sufficiently large $t$, we have
\[
\begin{split}
I'(t,x_n)
& =\left[\beta(t)S(t,x_n)-\eta(t)g(S(t,x_n),I(t,x_n),P(t,x_n))-c(t)\right]I(t,x_n) \\
& \ge - (\eta^ug(L,0,0) +c^u) I(t,x_n).
\end{split}
\]
Therefore we obtain
\[
\int_{s_{n,k}}^{t_{n,k}} \frac{I'(r,x_n)}{I(r,x_n)} \ dr \ge  - (\eta^u g(L,0,0) +c^u)  (t_{n,k}-s_{n,k})
\]
and thus $ I(t_{n,k},x_n) \ge I(s_{n,k},x_n) \e^{ - (\eta^u g(L,0,0) +c^u) (t_{n,k}-s_{n,k})}$.
By~\eqref{eq:maj-eps-eps1a}, and S6.3) we get
\[
\frac{\rho^u v(\eps/n^2)}{\rho^\ell v(\eps/n)} \ge
\frac{\rho(t_{n,k})v(\eps/n^2)}{\rho(s_{n,k})v(\eps/n)} \ge \e^{- (\eta^u g(L,0,0) +c^u) (t_{n,k}-s_{n,k})}
\]
and therefore we have
\begin{equation}\label{eq:limite-t-s}
t_{n,k}-s_{n,k}  \ge \frac{\log(\rho^\ell/\rho^u)+\log(v(\eps/n)/v(\eps/n^2))}{\eta^u g(L,0,0) +c^u} \to +\infty
\end{equation}
as $n \to +\infty$, since, by~{S6.2)} we have
\[
\lim_{n \to +\infty} \frac{v(\eps/n)}{v(\eps/n^2)} = \lim_{n \to +\infty} \frac{n\, v'(\eps/n)}{2\, v'(\eps/n^2)}
\ge \lim_{n \to +\infty}  \frac{An}{2B} =+\infty.
\]

By the first and third equations of~\eqref{eq:principal} and~\eqref{eq:maj-eps-eps2}, we have, for $t \in \, ]s_{n,k},t_{n,k}[$,
\small{
\[
\begin{cases}
S' \le {G_{2,\eps}(t,S(t,x_n))}\\
P'\le h_{2,\eps}(t,P(t,x_n))P(t,x_n)+\gamma(t)a(t)f(S(t,x_n),0,P(t,x_n))P(t,x_n)\\
\quad \quad \quad +\rho(t)v(\eps/n)\theta(t)\eta(t)g(S(t,x_n),{0},P(t,x_n))
\end{cases}.
\]}

Letting $(\hat x_{n,k}(t), \hat z_{n,k}(t))$ be the solution of
\[
\begin{cases}
{x' = G_{2,\eps}(t,x)}\\
z' = h_{2,\eps}(t,z)z+\gamma(t)a(t)f(x,0,z)z+\rho(t)v(\eps/n)\theta(t)\eta(t)g(x,{0},z)
\end{cases}
\]
with $\hat x_{n,k}(s_{n,k})=S(s_{n,k})$ and $\hat z_{n,k}(s_{n,k})=P(s_{n,k})$. We conclude that $S(t,x_n)\le \hat x_{n,k}(t)$ and $P(t,x_n)\le \hat z_{n,k}(t)$, for each $t \in \,
]s_{n,k},t_{n,k}[$.
By~{N{6.5})}, given $\delta>0$, we have
\[
\left|x^*_{2,\eps/n}(t)-\hat x_{n,k}(t)\right| < \delta/2 \quad \text{and} \quad \left|z^*_{2,\eps/n}(t)-\hat z_{n,k}(t)\right| < \delta/2,
\]
for all sufficiently large $k$ (that depends on $n$). By continuity, for sufficiently large $n$ and all sufficiently large $k \ge K(n)$, we have
\[
\left|{x^*_{2}}(t)-\hat x_{n,k}(t)\right|\le\left|{x^*_{2}}(t)-x^*_{2,\eps/n}(t)\right|+\left|x^*_{2,\eps/n}(t)-\hat x_{n,k}(t)\right|
\le \delta.
\]
and
\[
\left|{z^*_{2}}(t)-\hat z_{n,k}(t)\right|\le\left|{z^*_{2}}(t)-z^*_{2,\eps/n}(t)\right|+\left|z^*_{2,\eps/n}(t)-\hat z_{n,k}(t)\right|
\le \delta.
\]
In particular, for sufficiently large $n$, all sufficiently large $k\ge K(n)$ and for $t \in \, ]s_{n,k(n)},t_{n,k(n)}[$, we have
\begin{equation}\label{eq:bound-sist-persist-P-strog-persist}
S(t) \le \hat x_{n,k}(t) \le {x^*_{2}}(t)+\delta \quad \text{and} \quad P(t) \le \hat z_{n,k}(t) \le {z^*_{2}}(t)+\delta.
\end{equation}
Similar computations show that, for sufficiently large $n$, all sufficiently large $k\ge K(n)$ and for $t \in \, ]s_{n,k(n)},t_{n,k(n)}[$, we obtain
\begin{equation}\label{eq:bound-sist-persist-S-strog-persist}
S(t) \ge \tilde x_{n,k}(t) \ge {x^*_{1}}(t)-\delta \quad \text{and} \quad P(t) \ge \tilde z_{n,k}(t) \ge {z^*_{1}}(t)-\delta.
\end{equation}
Notice that, for a given $\delta$, eventually considering a larger $n$, we can take the same $n$ and $k$ in~{\eqref{eq:bound-sist-persist-P-strog-persist}} and~{\eqref{eq:bound-sist-persist-S-strog-persist}}.

Given $l>0$, by~\eqref{eq:limite-t-s} we can choose
$T > 0$  such that $t_{n,k}-s_{n,k} > l\lambda$ for all $n \ge T$. Therefore, by~\eqref{eq:maj-eps-eps1a},~{\eqref{eq:bound-sist-persist-P-strog-persist}} and~{\eqref{eq:bound-sist-persist-S-strog-persist}},
and by the second equation
in~\eqref{eq:principal}, for $n \ge T$ and $k\ge K(n)$ we get
\[
\begin{split}
& \frac{{\rho^u}v(\eps/n^2)}{(1+\theta^u\eta^u)(1+\beta^u)}\\
&  = I(t_{n,k},x_n) \\
& \ge I(s_{n,k},x_n) \exp\left\{\int_{s_{n,k}}^{t_{n,k}} \beta(r)S(r)-\eta(r)g(S(r),I(r),P(r))-c(r) \dr\right\} \\
& \ge I(s_{n,k},x_n) \times\\
& \phantom{\ge} \times \exp\left\{\kappa l+\int_{s_{n,k}+\lfloor(t_{n,k}-s_{n,k})/\lambda\rfloor}^{t_{n,k}} \beta(r)({x^*_{1}}(r)-\delta)
-\eta(r)g({x^*_{1}}(r)+\delta, \rho^u v(\eps/n), {z^*_{2}}(r)-\delta)-c(r) \dr\right\} \\
& \ge \frac{{\rho^\ell} v(\eps/n)}{(1+\theta^u\eta^u)(1+\beta^u)}
\e^{\kappa l-\lambda(\beta^u\delta+\eta^u(g(({x^*_{1}})^u-\delta, \rho^u v(\eps/n),({z^*_{2}})^u+\delta)+c^u)}\\
& > \frac{{\rho^\ell}v(\eps/n)}{(1+\theta^u\eta^u)(1+\beta^u)} ,
\end{split}
\]
for sufficiently large $l$ (that requires that $T$ is sufficiently large). We conclude that
\[
\frac{\rho^u v(\eps/n^2)}{\rho^\ell v(\eps/n)}  {>} 1
\]
and this contradicts the fact that, by S6.2) and S6.3), we have
\[
\lim_{n \to +\infty} \frac{\rho^u v(\eps/n^2)}{\rho^\ell v(\eps/n)} = \lim_{n \to +\infty} \frac{2\rho^u v'(\eps/n^2)/n^3}{\rho^\ell v'(\eps/n)/n^2}
\le \lim_{n \to +\infty} \frac{2\rho^u {B}}{n\rho^\ell {A}}=0.
\]

We conclude that there is $m_1>0$ such that $\displaystyle \liminf_{t \to +\infty} I(t) > m_1$ and the result follows from~{S\ref{cond-3})}.
\end{proof}

{In~{\cite{Lu-Wang-Liu-DCDS-B-2018}}, the authors obtain extinction and persistence results for eco-epidemiological model with Crowley-Martin functional response. In the
extinction result the authors consider auxiliary equations different from~\eqref{eq:auxiliary-S(t)} and~\eqref{eq:auxiliary-P(t)} using some upper bound for $S$ and some lower
bound for $P$ related to the dimension of some positive invariant region that contains the omega limit of all solutions.
We will borrow and improve the idea of that paper in our context. To this purpose, we need to consider families of auxiliary equations.
We begin by noticing that, by the proof of Theorem~\ref{teo:main-extinction}, for that any solution $(S(t),I(t),P(t))$ of our problem with initial condition
$(S(t_0),I(t_0),P(t_0))=(S_0,I_0,P_0)$ we have $s^{1,\ell}(t)\le S(t)\le s^{1,u}(t)$
and $y^{1,\ell}(t)\le P(t)\le y^{1,u}(t)$, for all $t>0$ sufficiently large, where $s^{1,\ell}(t)=0$, $s^{1,u}(t)$ is the solution of~\eqref{eq:auxiliary-S(t)} with initial
condition $s^{1,u}(t_0)=S_0$, $y^{1,\ell}(t)$ is the solution of~\eqref{eq:auxiliary-P(t)} with initial condition $y^{1,\ell}(t_0)=P_0$ and $y^{1,u}(t)=L$, where $L$ is given in
condition~S\ref{cond-12}). Consider the equations:}
{
\begin{equation}\label{eq:auxiliary-S(t)-1}
s'=G(t,s)-a(t)f(s,L,y^{1,u}(t))y^{1,\ell}(t),
\end{equation}
and
\begin{equation}\label{eq:auxiliary-S(t)-2}
s'=G(t,s)-a(t)f(s,0,y^{1,\ell}(t))y^{1,u}(t)-\beta(t)sL,
\end{equation}
where $y^{1,\ell}(t)$ is a particular solution of~\eqref{eq:auxiliary-P(t)}.
For equations~\eqref{eq:auxiliary-S(t)-1} and~\eqref{eq:auxiliary-S(t)-2}, we assume the following:}
{
\begin{enumerate}[S$1$')]
\setcounter{enumi}{3}
\item \label{cond-9-1} Each solution $s(t)$ of~\eqref{eq:auxiliary-S(t)-1} (respectively \eqref{eq:auxiliary-S(t)-2}) with positive initial condition is bounded, bounded away
    from zero, and globally attractive on $]0,+\infty[$, that is $|s(t)-v(t)| \to 0$ as $t \to +\infty$ for each solution $v(t)$ of~\eqref{eq:auxiliary-S(t)-1} (respectively
    \eqref{eq:auxiliary-S(t)-2}) with positive initial condition.
\end{enumerate}
}
{
We also need to consider the equations
\begin{equation}\label{eq:auxiliary-P(t)-1}
y'=h(t,y)y+\gamma(t)a(t)f(s^{2,\ell}(t),L,y)y,
\end{equation}
and
\begin{equation}\label{eq:auxiliary-P(t)-2}
y'=h(t,y)y+\gamma(t)a(t)f(s^{2,u}(t),0,y)y+\theta(t)\eta(t)g(0,0,y)L,
\end{equation}
where $s^{2,u}(t)$ is a particular solution of~\eqref{eq:auxiliary-S(t)-2}.
For the family of equations~\eqref{eq:auxiliary-P(t)-1}, we assume the following:}
{
\begin{enumerate}[S$1$')]
\setcounter{enumi}{4}
\item \label{cond-10-1} Each fixed solution $y(t)$ of~\eqref{eq:auxiliary-P(t)-1} (respectively~\eqref{eq:auxiliary-P(t)-2}) with positive initial condition is bounded and
    globally attractive on $[0,+\infty)$.
\end{enumerate}
}
{Using the solutions of the systems above we can define the following number:
\begin{equation}\label{eq:cR0-per-linha}
\cR^{u,1}(\lambda)=\limsup_{t \to +\infty} \int_t^{t+\lambda} \beta(s)s^\sharp(s)-\eta(s)g(s^\sharp(s),0,y^\sharp(s))-c(s) \ds.
\end{equation}
where $s^\sharp(t)$ and $y^\sharp(t)$ are particular solutions, respectively, of~\eqref{eq:auxiliary-S(t)-1} and~\eqref{eq:auxiliary-P(t)-1} with positive initial conditions.
Notice that, according to our assumptions, it is easy to prove, with similar arguments to the ones in Proposition~\ref{teo:independence-of-R0-of-solutions}, that
$\cR^{u,1}(\lambda)$ is independent of the particular solutions considered.
\begin{theorem}\label{teo:main-extinction-2}
Assume that $g(S+I,0,P)\le g(S,I,P)$ and that conditions {S\ref{cond-1})} to {S\ref{cond-3})}, S\ref{cond-9-1}') and S\ref{cond-10-1}') hold. Assume further that $G(t,S)=\Lambda(t)-\mu(t)S$ or that $g$ does not depend on $I$. If there is $\lbd>0$ such that
$\cR^{u,1}(\lambda)<0$, then the infectives in system~\eqref{eq:principal} go to extinction.
\end{theorem}
\begin{proof}
The proof consists in repeating the steps in the proof of Theorem~\ref{teo:main-extinction}, with the changes that we will describe below. In the first place, instead of the
bounds given by \eqref{eq:auxiliary-S(t)} and \eqref{eq:auxiliary-P(t)}, we use bounds obtained in the following way: letting $y^{1,\ell}(t)$ and $y^{2,u}(t)$ be the solutions
defined above, we know that
\begin{equation}\label{teo-eq-aux-1st-iteration-a}
S'\le G(t,S)-a(t)f(S,L,y^{1,u}(t))y^{1,\ell}(t)
\end{equation}
and
\begin{equation}\label{teo-eq-aux-1st-iteration-b}
S'\ge G(t,S)-a(t)f(S,0,y^{1,\ell}(t))y^{1,u}(t)-\beta(t)SL.
\end{equation}
Thus, using the solutions $s^{2,\ell}(t)$ and $s^{2,u}(t)$ respectively of~\eqref{teo-eq-aux-1st-iteration-a} and~\eqref{teo-eq-aux-1st-iteration-b}, we obtain
\begin{equation}\label{teo-eq-aux-1st-iteration-c}
P'\ge h(t,P)P+\gamma(t)a(t)f(s^{2,\ell}(t),L,P)P
\end{equation}
and
\begin{equation*}
P' \le h(t,P)P+\gamma(t)a(t)f(s^{2,u}(t),0,P)P+\theta(t)\eta(t)g(0,0,P)L.
\end{equation*}
The bounds in~\eqref{teo-eq-aux-1st-iteration-a} and~\eqref{teo-eq-aux-1st-iteration-c} allow us to conclude that, for sufficiently large $t>0$, we have $S(t)\le s^\sharp(t)$ and
$P(t)\ge y^\sharp(t)$, where $s^\sharp(t)$ and $y^\sharp(t)$ are respectively particular solutions of \eqref{eq:auxiliary-S(t)-1} and \eqref{eq:auxiliary-P(t)-1}; using the
solutions $s^\sharp(t)$ and $y^\sharp(t)$ and the number $\cR^{u,1}(\lambda)$ in~\eqref{eq:cR0-per-linha}, similar arguments to the ones in Theorem~\ref{teo:main-extinction}
allow us to obtain the result.
\end{proof}
}
{We note that the procedure in Theorem~\ref{teo:main-extinction-2} can be iterated to obtain new and (hopefully) better estimates of the region of extinction, as long as we
can still ensure that assumptions S\ref{cond-9-1}') and S\ref{cond-10-1}') still hold for the new equations. In fact all we have to do is the following: consider equations
\eqref{eq:auxiliary-S(t)-1} and \eqref{eq:auxiliary-S(t)-2} with $y^{1,\ell}(t)$ and $y^{1,u}(t)$ replaced by $y^{2,\ell}(t)$ and $y^{2,u}(t)$, the solutions of
\eqref{eq:auxiliary-P(t)-1} and \eqref{eq:auxiliary-P(t)-2}, and denote the solutions of those equations by $s^{3,\ell}(t)$ and $s^{3,u}(t)$; consider equations
\eqref{eq:auxiliary-P(t)-1} and \eqref{eq:auxiliary-P(t)-2} with $s^{2,\ell}(t)$ and $s^{2,u}(t)$ replaced by $s^{3,\ell}(t)$ and $s^{3,u}(t)$; replace $\cR^{u,1}(\lambda)$ by
$$\cR^{u,2}(\lambda)=\limsup_{t \to +\infty} \int_t^{t+\lambda} \beta(s)s^{\sharp\sharp}(s)-\eta(s)g(s^{\sharp\sharp}(s),0,y^\sharp(s))-c(s) \ds,$$
where $s^{\sharp\sharp}(t)$ and $y^{\sharp\sharp}(t)$ are particular solutions, respectively, of the new equations corresponding to~\eqref{eq:auxiliary-S(t)-1}
and~\eqref{eq:auxiliary-P(t)-1} with positive initial conditions. With these ingredients we obtain a new theorem on extinction. As long as the assumptions corresponding to
S\ref{cond-9-1}') and S\ref{cond-10-1}') still hold, we can repeat the process over and over again obtaining a sequence of theorems on extinction and (hopefully) improving the
estimates at each step.
}

\section{Examples}\label{examples}

{In this section we will apply Theorem~\ref{teo:main-extinction} and Theorem~\ref{teo:main-persistence} to some particular cases of model~\eqref{eq:principal}}.

\subsection{Models with no predation on uninfected preys}\label{example-1}
In this section we will consider a family of models with no predation on uninfected preys by letting $f \equiv 0$ and $g(S,I,P)=P$. This family generalises the family of models in~\cite{Niu-Zhang-Teng-AMM-2011} by allowing a very general form for the vital dynamics of predators and preys.  Thus,  we consider in this subsection the following model:
\begin{equation*}\label{eq:modelo-Niu-Zhang-Teng}
\begin{cases}
S'=G(t,S)-\beta(t)SI\\
I'=\beta(t)SI-\eta(t)PI-c(t)I\\
P'=h(t,P)P+\theta(t)\eta(t)PI
\end{cases}.
\end{equation*}
In this context,~\eqref{eq:cR0-ext} and~\eqref{eq:cR0-per} become
\begin{equation*}
\cR_{np}^\ell(\lambda)=\liminf_{t \to +\infty} \int_t^{t+\lambda} \beta(s)s^*(s)-\eta(s)y^*(s)-c(s) \ds
\end{equation*}
and
\begin{equation*}
\cR_{np}^u(\lambda)=\limsup_{t \to +\infty} \int_t^{t+\lambda} \beta(s)s^*(s)-\eta(s)y^*(s)-c(s) \ds.
\end{equation*}
where $s^*(t)$ and $y^*(t)$ are particular solutions, respectively, of~\eqref{eq:auxiliary-S(t)} and~\eqref{eq:auxiliary-P(t)}.

Under the hypotheses of Theorem~\ref{teo:main-extinction}, we obtain that if there is $\lbd>0$ such that $\cR^u_{np}(\lambda)<0$ then the infectives in system~\eqref{eq:modelo-Niu-Zhang-Teng} go to extinction, and under the hypotheses of Theorem~\ref{teo:main-persistence}, we conclude that if there is $\lbd>0$ such that $\cR^\ell_{np}(\lambda)>0$ then the infectives in system~\eqref{eq:modelo-Niu-Zhang-Teng} are uniform strong persistent.

As we already mentioned, model~\eqref{eq:modelo-Niu-Zhang-Teng} includes the model discussed in~\cite{Niu-Zhang-Teng-AMM-2011} as the particular case where $G(t,S)=\Lambda(t)-\mu(t)S$ and $h(t,P)=b(t)-r(t)P$, with $\Lambda$, $\mu$, $r$ and $b$ nonnegative, continuous and bounded functions satisfying:
\[
\liminf_{t \to +\infty} \int_t^{t+\omega_1} \Lambda(s)\ds>0, \quad \quad
\liminf_{t \to +\infty} \int_t^{t+\omega_2} \mu(s)\ds>0,
\]
\[
\liminf_{t \to +\infty} \int_t^{t+\omega_3} r(s)\ds>0 \quad \quad \text{and} \quad \quad
\liminf_{t \to +\infty} \int_t^{t+\omega_4} b(s)\ds>0,
\]
for some constants $w_i>0$, $i=1,\ldots,4$:
\begin{equation}\label{eq:modelo-Niu-Zhang-Teng-2}
\begin{cases}
S'=\Lambda(t)-\mu(t)S-\beta(t)SI\\
I'=\beta(t)SI-\eta(t)PI-c(t)I\\
P'=(b(t)-r(t)P)P+\theta(t)\eta(t)PI
\end{cases}.
\end{equation}
Note that, for the model in~\eqref{eq:modelo-Niu-Zhang-Teng-2}, condition S\ref{cond-1}) is assumed, condition~S\ref{cond-3}) is immediate from the particular forms of the functions $g$ and $h$, conditions~S\ref{cond-9}) and S\ref{cond-10}) follow from Lemmas 1 and 3 in~\cite{Niu-Zhang-Teng-AMM-2011} and condition~S\ref{cond-11}) is a consequence of the fact that, in this setting, systems~\eqref{eq:auxiliary-system-SP-1} and~\eqref{eq:auxiliary-system-SP-2} are uncoupled and small perturbations of each of the equations in those systems is globally asymptotically stable by Lemmas 1 and 3 in~\cite{Niu-Zhang-Teng-AMM-2011}. Finally, condition~S\ref{cond-12}) follows from Theorem 1 in~\cite{Niu-Zhang-Teng-AMM-2011}. We also note that $\cR_{np}^u(\lambda)$ and $\cR_{np}^\ell(\lambda)$ coincide with the corresponding numbers in~\cite{Niu-Zhang-Teng-AMM-2011}.

Another possible choice for the functions $h$ and $G$ is $h(t,P)=-(\delta_1(t)+\delta_2(t)P)$, with $\delta_1$ and $\delta_2$ continuous and nonnegative functions and $G(t,S)=k(t,S)S$ with $k$ a continuous and bounded function satisfying the conditions: $\partial k/\partial S(t,s)<0$, for every $t,s\ge 0$; $k(t,0)>0$ for all $t \ge 0$; there is $S_1(t)>0$ such that $k(t,S_1(t))=0$, for every $t \ge 0$. This choice makes the underlying predator-uninfected prey subsystem in model~\eqref{eq:modelo-Niu-Zhang-Teng} correspond to the model studied in section 3 of~\cite{Garrione-Rebelo-NARWA-2016} with the function $f \equiv 0$.
System~\eqref{eq:modelo-Niu-Zhang-Teng} becomes in this case:
\begin{equation*}
\begin{cases}
S'=k(t,S)S-\beta(t)SI\\
I'=\beta(t)SI-\eta(t)PI-c(t)I\\
P'=-(\delta_1(t)+\delta_2(t)P)P+\theta(t)\eta(t)PI
\end{cases}.
\end{equation*}
Notice that the study of the function $k(t,S)$ in~\cite{Garrione-Rebelo-NARWA-2016} allow us to conclude easily that conditions~S\ref{cond-1}) to~S\ref{cond-10}) are satisfied for this model. Condition~S\ref{cond-11}) is a consequence of the fact that systems~\eqref{eq:auxiliary-system-SP-1} and~\eqref{eq:auxiliary-system-SP-2} are uncoupled and small perturbations of each of the equations in those systems is globally asymptotically stable (the global asymptotic stability of the first equation is consequence of Lemma 3.1 in~\cite{Garrione-Rebelo-NARWA-2016} and the global asymptotic stability of the second equation is trivial).

\,

To do some simulation, in this scenario we assumed that $G(t,S)=(0.7-0.6S)S$; $\beta(t)=\beta_0(1+0.7\cos(2\pi t))$; $\eta(t)=0.7(1+0.7\cos(\pi+2\pi t))$; $c(t)=0.1$; $h(t,P)=-0.2-0.3P$; $\theta(t)=0.9$. We obtain the model:
\begin{equation*}
\begin{cases}
S'=(0.7-0.6S)S-\beta_0(1+0.7\cos(2\pi t))SI\\
I'=\beta_0(1+0.7\cos(2\pi t))SI-0.7(1+0.7\cos(\pi+2\pi t))PI-0.1I\\
P'=(-0.2-0.3P)P+0.63(1+0.7\cos(\pi+2\pi t))PI
\end{cases},
\end{equation*}

When $\beta_0=0.01$ we obtain $\cR^u= -0.15<0$ and we conclude that we have extinction (figure~\ref{fig_exe1_extincao}). When $\beta_0=0.3$ we obtain  $\cR^\ell=1.3>0$ and we conclude that the infectives are uniform strong persistent (figure~\ref{fig_exe1_persistencia}).

We considered the following initial conditions at $t=0$: $(S_0,I_0,P_0)=(1,0.5,0.1)$, $(S_0,I_0,P_0)=(0.1,0.2,1)$ and $(S_0,I_0,P_0)=(0.5,0.5,0.5)$. 

\begin{figure}[H]
  \begin{minipage}[b]{.32\linewidth}
    \includegraphics[width=\linewidth]{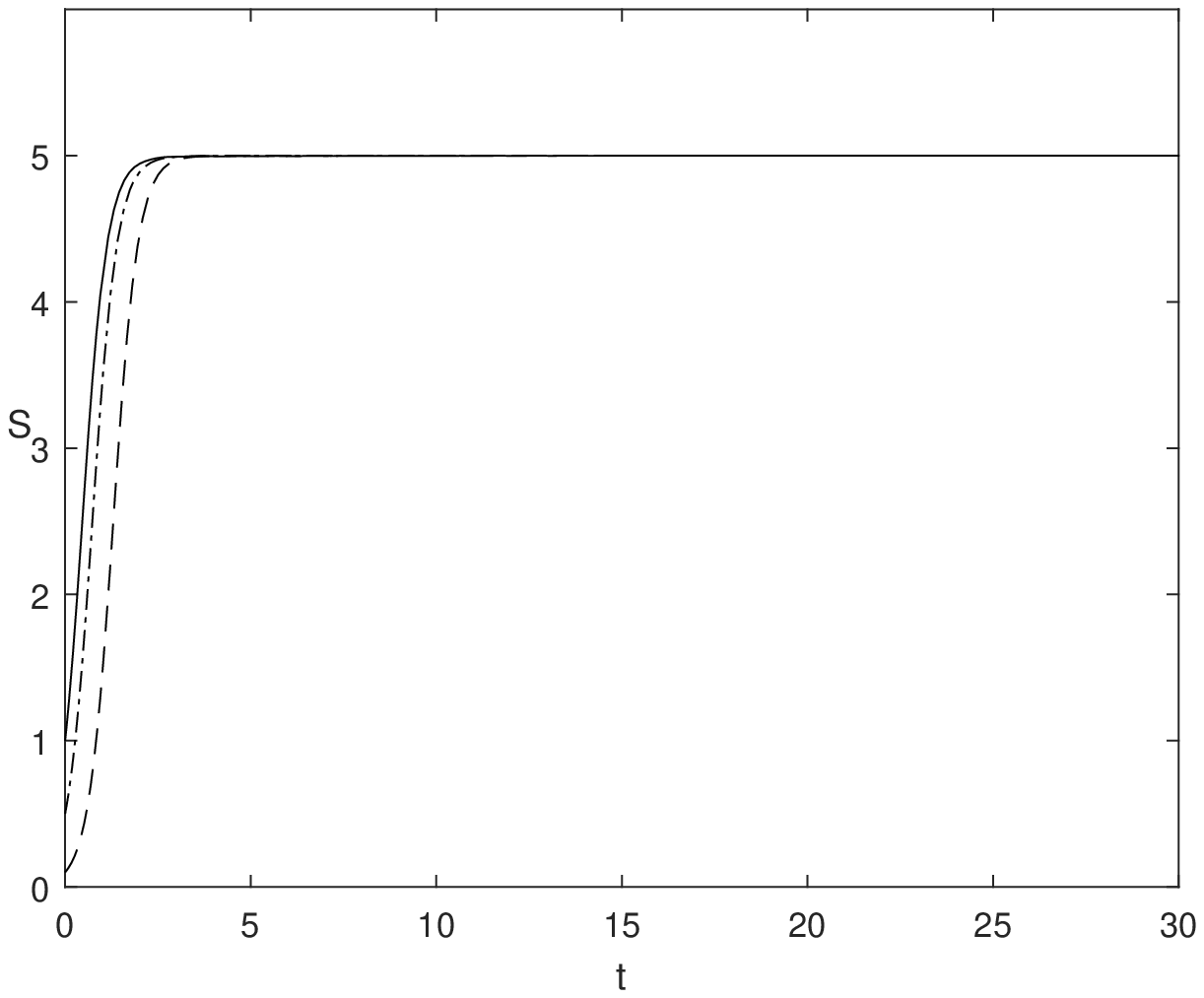}
  \end{minipage}
  \begin{minipage}[b]{.32\linewidth}
        \includegraphics[width=\linewidth]{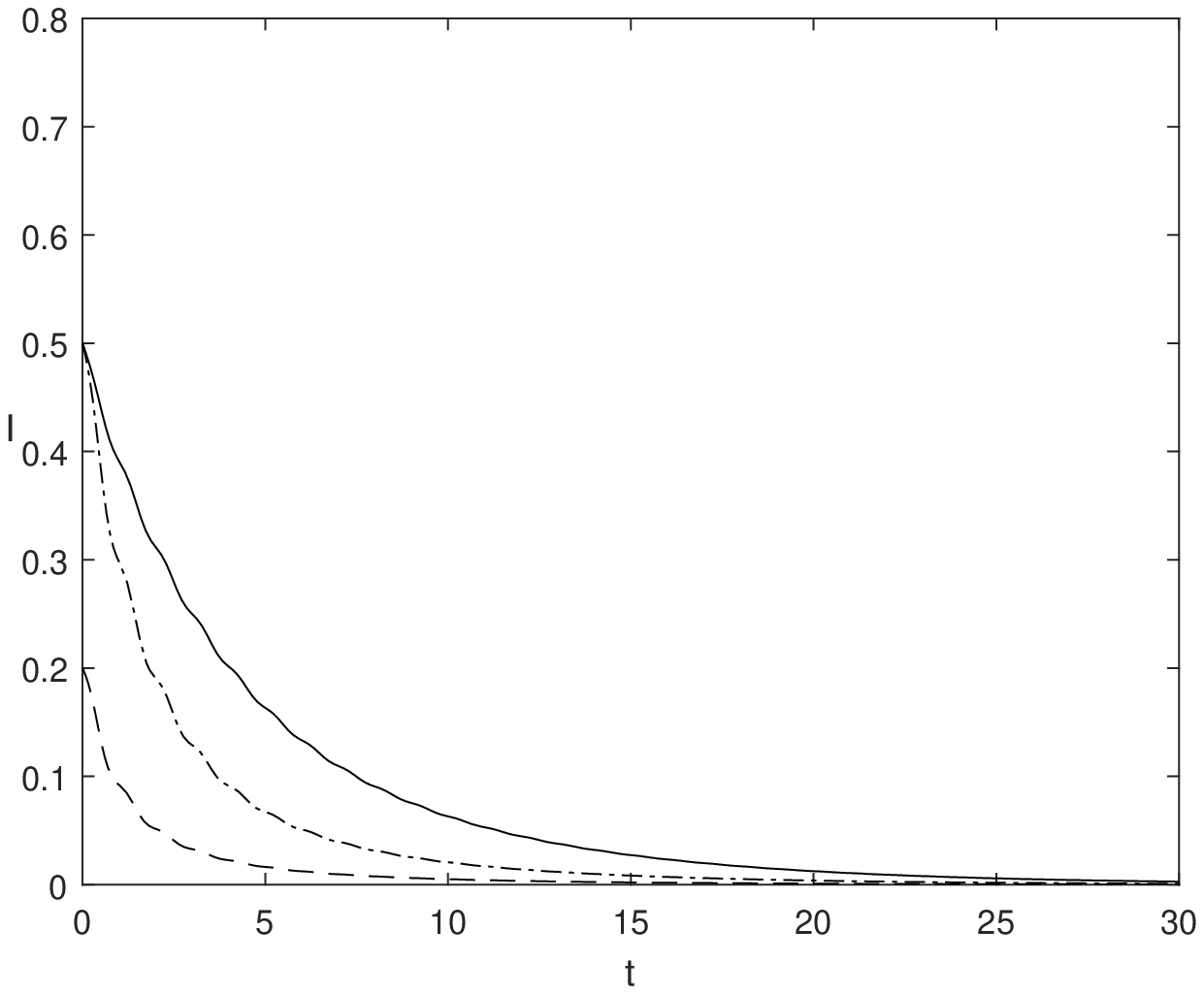}
  \end{minipage}
  \begin{minipage}[b]{.32\linewidth}
        \includegraphics[width=\linewidth]{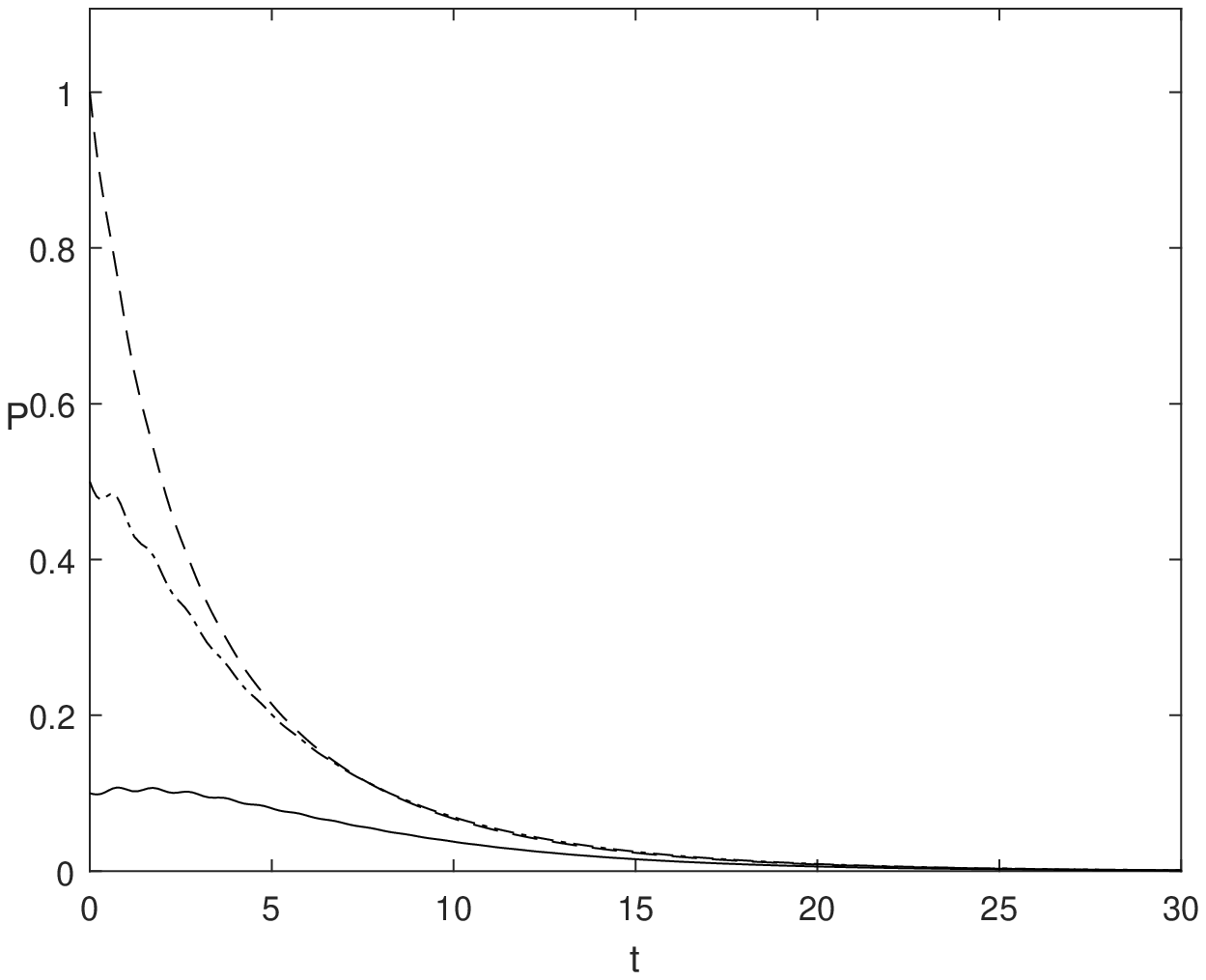}
  \end{minipage}
    \caption{Extinction: $\beta_0=0.01$.}
      \label{fig_exe1_extincao}
\end{figure}

\begin{figure}[H]
  \begin{minipage}[b]{.32\linewidth}
    \includegraphics[width=\linewidth]{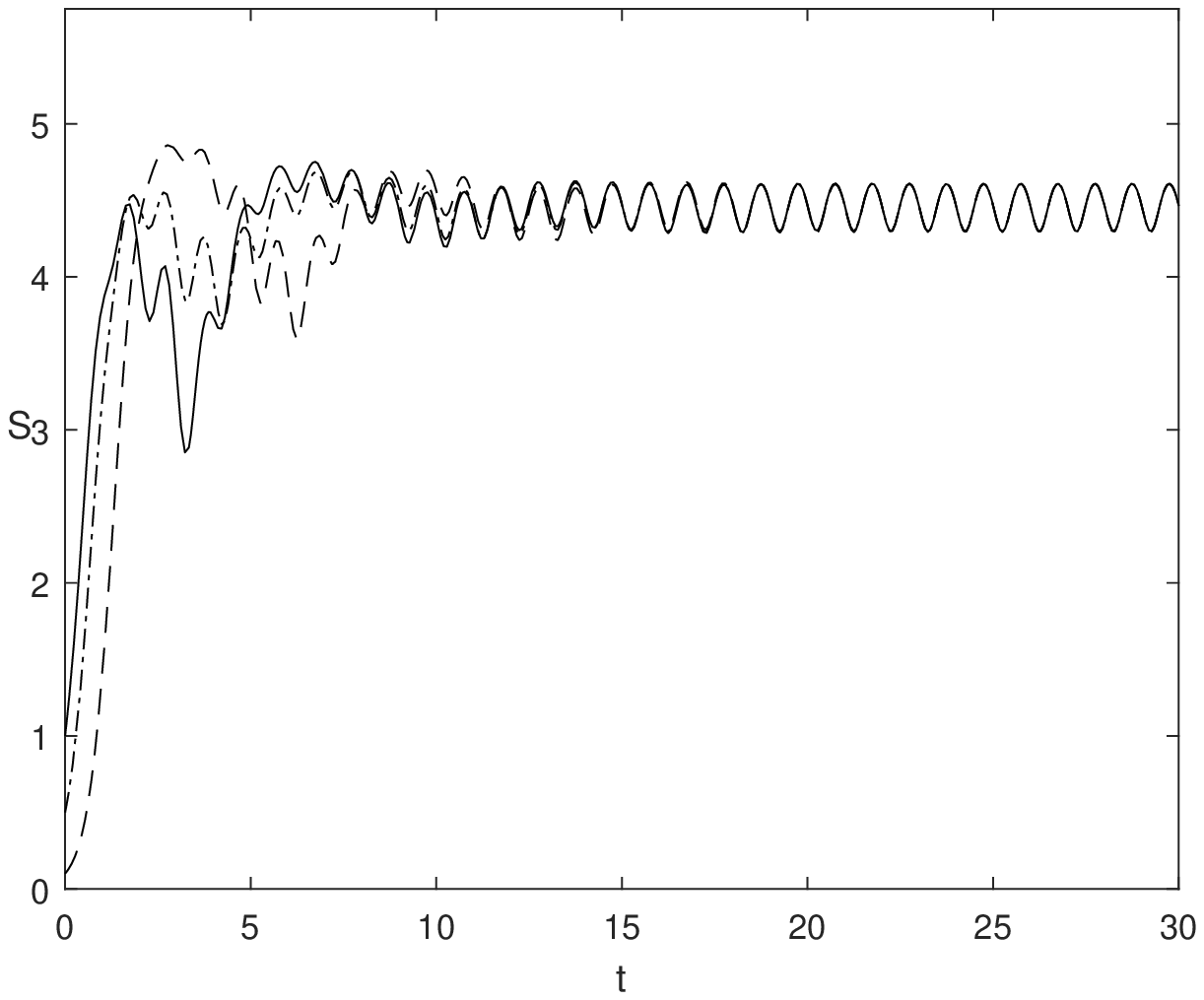}
  \end{minipage}
  \begin{minipage}[b]{.32\linewidth}
        \includegraphics[width=\linewidth]{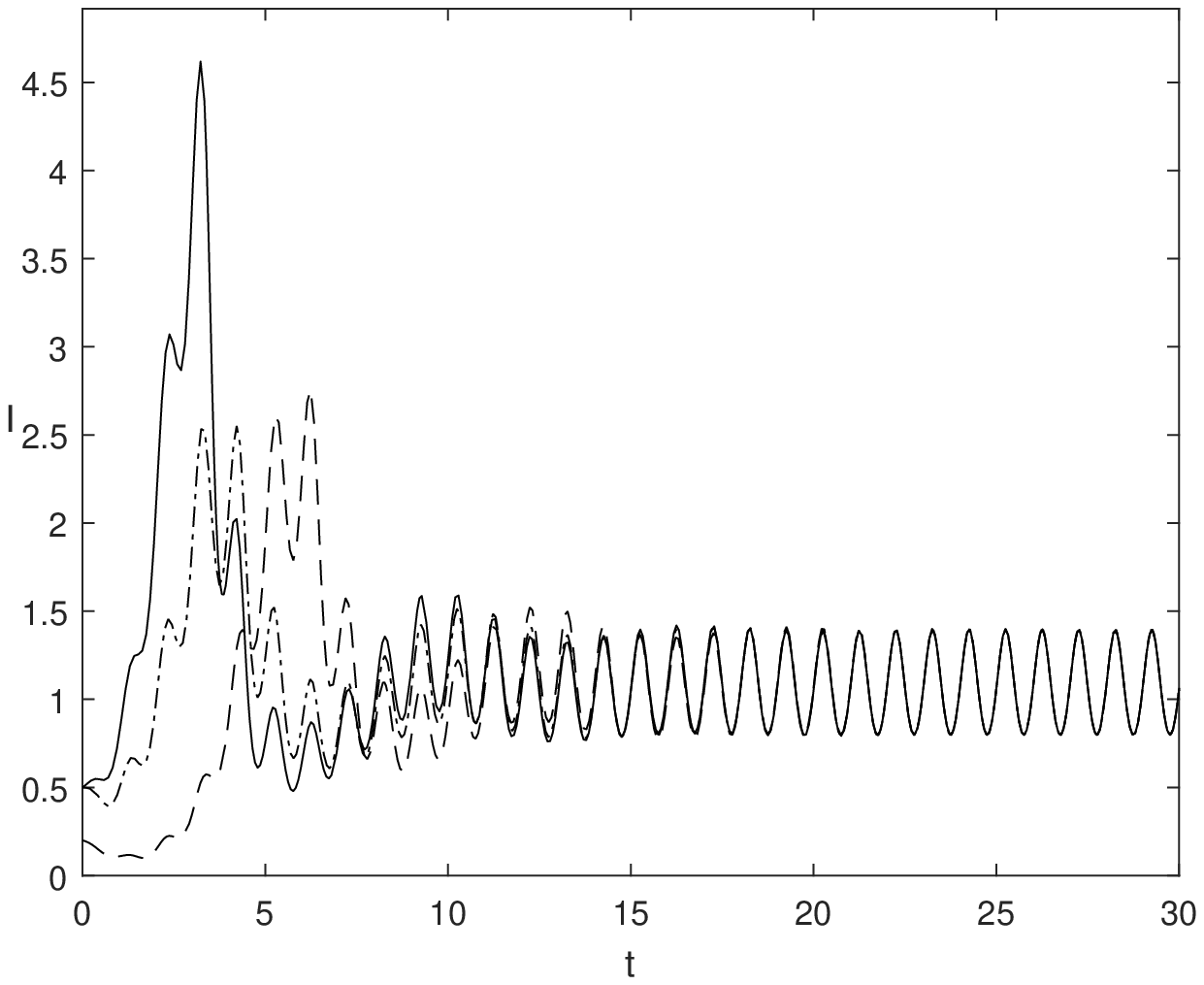}
  \end{minipage}
  \begin{minipage}[b]{.32\linewidth}
        \includegraphics[width=\linewidth]{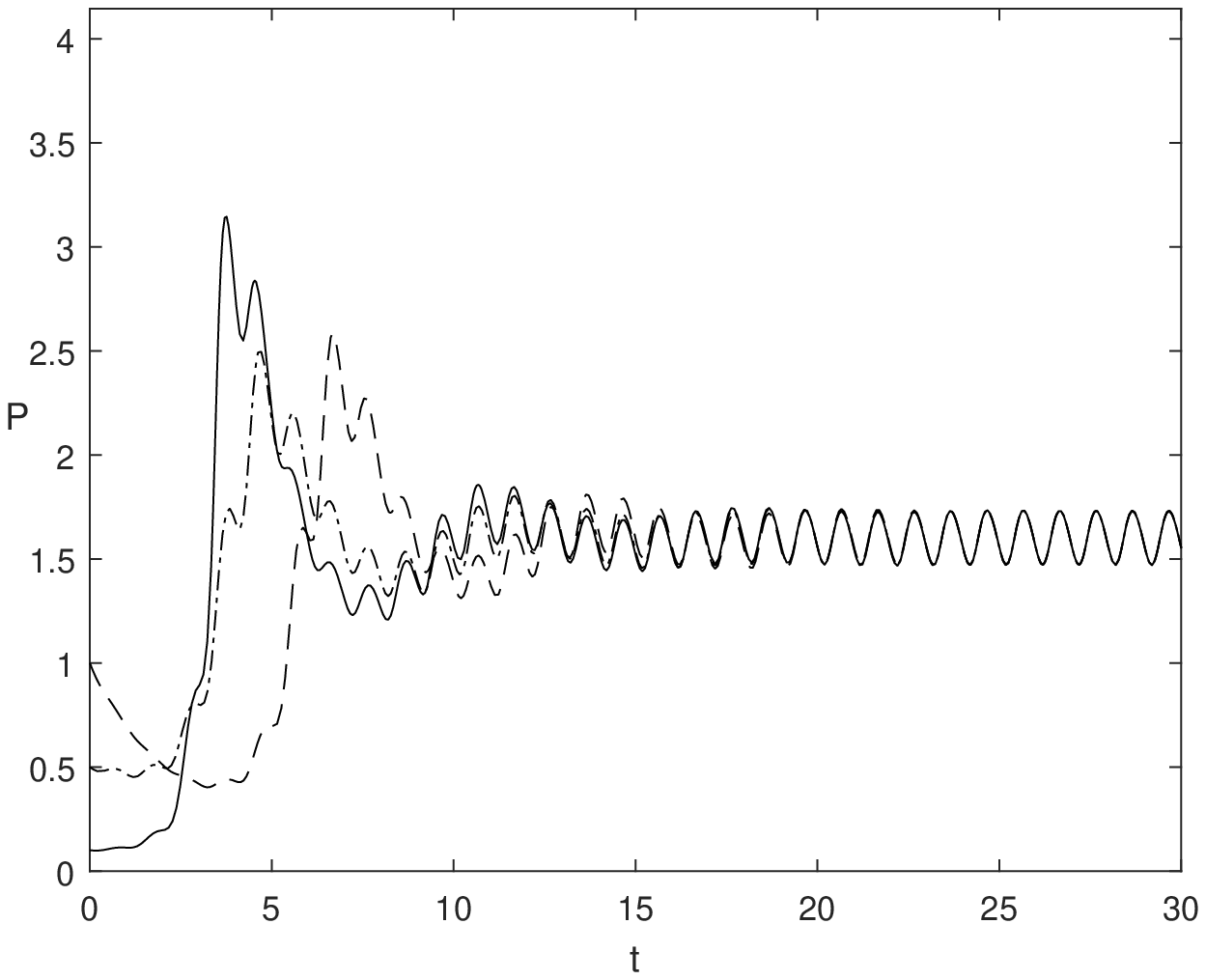}
  \end{minipage}
    \caption{Uniform strong persistence: $\beta_0=0.3$.}
      \label{fig_exe1_persistencia}
\end{figure}

\subsection{Periodic coefficients}
In this subsection we consider a family of models with periodic parameters and predation on uninfected preys that, in general, is not included in the general family of models
considered in~\cite{Lu-Wang-Liu-DCDS-B-2018}. For periodic models, the thresholds become easier to deal with.

Assume that there is $\omega >0$ such that all parameters in~\eqref{eq:principal} are $\omega$-periodic functions. In this case,~\eqref{eq:cR0-ext} and~\eqref{eq:cR0-per}
become, respectively,
\[
\cR^\ell(\omega)=\int_0^{\omega} \beta(s){x^*_{1}}(s)-\eta(s)g({x^*_{1}}(s),0,{z^*_{2}}(s))-c(s) \ds,
\]
and
\[
\cR^u(\omega)= \int_0^{\omega} \beta(s)s^*(s)-\eta(s)g({s^*}(s),0,{y^*}(s))-c(s) \ds.
\]
Thus
\[
\cR^\ell(\omega)>0 \quad \Leftrightarrow \quad \frac{\overline{\beta {x^*_{1}}}}{\overline{\eta g({x^*_{1}},0,{z^*_{2}})}+\overline{c}}>1
\]
and
\[
\cR^u(\omega)<0 \quad \Leftrightarrow \quad \frac{\overline{\beta s^*}}{\overline{\eta g(s^*,0,y^*)}+\overline{c}}<1.
\]
where $s^*(t)$ and $y^*(t)$ are particular solutions, respectively, of~\eqref{eq:auxiliary-S(t)} and~\eqref{eq:auxiliary-P(t)}, and {${x^*_{1}}(t)$ and ${z^*_{2}}(t)$ still denote any particular solution of first and second equations in systems~\eqref{eq:auxiliary-system-SP-1} and~\eqref{eq:auxiliary-system-SP-2}, respectively, with positive initial conditions}. Define
\[
\cR^\ell_{per}=\frac{\overline{\beta {x^*_{1}}}}{\overline{\eta g({x^*_{1}},0,{z^*_{2}})}+\overline{c}}
\quad \quad \text{and}
\quad \quad
\cR^u_{per}=\frac{\overline{\beta s^*}}{\overline{\eta g(s^*,0,y^*)}+\overline{c}}.
\]
Under the hypotheses of Theorem~\ref{teo:main-extinction}, we have that if $\cR^u_{per}<1$ then the infectives in model~\eqref{eq:principal} with periodic coefficients go to extinction, and
under the hypotheses of Theorem~\ref{teo:main-persistence}, if $\cR^\ell_{per}>1$ then the infectives  in model~\eqref{eq:principal} with periodic coefficients are uniform strong persistent.

Note that the corollaries in~\cite{Niu-Zhang-Teng-AMM-2011}, concerning the periodic case, are particular cases of the corollaries above.
In fact, in~\cite{Niu-Zhang-Teng-AMM-2011} we have $f \equiv 0$ and in this case, as argued in the previous section, $(s^*(t),y^*(t))$ is a particular solution
of~\eqref{eq:uninfected-system}, condition {S\ref{cond-1})} is assumed, condition~{S\ref{cond-3})} is immediate, conditions~~{S\ref{cond-12})} to
{S\ref{cond-11})} follow from results in~\cite{Niu-Zhang-Teng-AMM-2011}. Thus, when $f\equiv 0$, we get similar thresholds to the ones in the mentioned paper:
$$\cR^\ell_{per}=\cR^u_{per}=\frac{\overline{\beta s^*}}{\overline{\eta y^*}+\overline{c}}.$$
This threshold can also be obtained using the procedures in~\cite{Rebelo-Margheri-Bacaer-JMB-2012,Wang-Zhao-JDDE-2008}.

We will focus now on a particular models with a function $G$ that is different from the corresponding function in~\cite{Lu-Wang-Liu-DCDS-B-2018}. We consider the following setting:
$G(t,S)=(\Lambda-\mu S)S$; $a(t)=a$; $f(S,I,P)=S$; $g(S,I,P)=P$; $h(t,P)=b-rP$; $\gamma(t)=\gamma$. We obtain the model:
\begin{equation}\label{eq:Periodic-first-scenario}
\begin{cases}
S'=(\Lambda-\mu S)S- a SP-\beta(t)SI\\
I'=\beta(t)SI-\eta(t)PI-c(t)I\\
P'=(b-rP)P+ \gamma a SP+\theta(t)\eta(t)PI
\end{cases},
\end{equation}
For this model, condition {S\ref{cond-1})} is assumed, condition~{S\ref{cond-3})} is immediate from the particular forms of the functions $g$ and $h$,
conditions~{S\ref{cond-9})} and {S\ref{cond-10})} hold for our particular functions as already discussed in section~\ref{example-1}. {In this context, an endemic
equilibrium for \eqref{eq:auxiliary-system-SP-2} is
$
\left(\Lambda/\mu,\hat{z}_\epsilon\right),
$
with $\hat{z}_\epsilon=({b\mu+a\gamma\Lambda+\epsilon\mu})/{\mu r}$, and  the endemic
equilibrium for \eqref{eq:auxiliary-system-SP-1} exists if $\Lambda r>ab+a\gamma\Lambda/\mu$:
$$
\left(\frac{\hat\Lambda}\mu,\frac{b\mu+a\gamma\hat\Lambda}{\mu r}\right),
$$
with $\hat\Lambda=\Lambda-a\hat z_\epsilon-\epsilon$.
}
These  subsystems can be discussed using~\cite{Goh-AN-1976}. In fact, the global asymptotic stability result proved in section 3 of~\cite{Goh-AN-1976} implies that, if
{$\Lambda r>ab+a\gamma\Lambda/\mu$}, condition~{S\ref{cond-11})} is satisfied. Finally, condition~{S\ref{cond-12})} is consequence of the following lemma:
\begin{lemma}
There is a bounded region that contains the $\omega$-limit of all orbits of~\eqref{eq:Periodic-first-scenario}.
\end{lemma}
\begin{proof}
Let $\eps>0$. Since, by the first equation in~\eqref{eq:Periodic-first-scenario}, $S' \le (\Lambda-\mu S)S$, we conclude that
\begin{equation}\label{eq:example-2-1}
S(t)\le \frac{\Lambda}{\mu} + \eps,
\end{equation}
for all $t$ sufficiently large. Additionally, we get
\begin{equation}\label{eq:example-2-2}
\sup_{S \in \mathbb{R}} (\Lambda-\mu S)S \le \left(\Lambda-\frac{\mu\Lambda}{2\mu}\right)\frac{\Lambda}{2\mu}=\frac{\Lambda^2}{4\mu}.
\end{equation}
Adding the first two equations in~\eqref{eq:Periodic-first-scenario} and using~\eqref{eq:example-2-1} and~\eqref{eq:example-2-2} we have, for all $t$ sufficiently large,
\begin{equation}\label{eq:example-2-3}
\begin{split}
(S+I)'
& =(\Lambda-\mu S)S-c(t)I\\
& \leqslant \frac{\Lambda^2}{4\mu}+c(t)S-c(t)(S+I)\\
& \leqslant \frac{\Lambda^2}{4\mu}+c^u \frac{\Lambda}{\mu} + c^u \eps -c^\ell(S+I).
\end{split},
\end{equation}
Since $\eps>0$ is arbitrary, we conclude that
\begin{equation*}
\limsup_{t \to +\infty} (S+I)(t) \le \frac{1}{c^\ell}\left(\frac{\Lambda^2}{4\mu}+c^u \frac{\Lambda}{\mu}\right):=A.
\end{equation*}
Finally, by the third equation in~\eqref{eq:Periodic-first-scenario} and~\eqref{eq:example-2-3}, given $\eps>0$, we get
\begin{equation}\label{Eq:2.K}
\begin{split}
P'
& =(b-rP)P+ \gamma a SP+\theta(t)\eta(t)PI\\
& \le \left(b+\gamma a A+\theta^u\eta^u A - rP \right)P,
\end{split}
\end{equation}
for sufficiently large $t$. Thus,
\[
\limsup_{t \to +\infty} P(t) \leqslant \frac{1}{r}\left(b+\gamma a A+\theta^u\eta^u A\right):=B.
\]
Equations~\eqref{eq:example-2-3} and~\eqref{Eq:2.K} show that the region
$$\{(S,I,P) \in \R^3: 0 \le S+I \le A \ \text{and} \ 0 \le P \le B\}$$
contains the $\omega$-limit of any orbit. 
\end{proof}

\,

To do some simulation, in this scenario we assumed that $G(t,S)=(0.7-0.6S)S$; $a=0.9$; $\beta(t)=\beta_0(1+0.7\cos(2\pi t))$; $\eta(t)=0.7(1+0.7\cos(\pi+2\pi t))$; $c(t)=0.1$; $b=0.2$; $r=0.6$; $\gamma=0.1$; $\theta(t)=0.9$. We obtain the model:
\begin{equation*}
\begin{cases}
S' =(0.7-0.6S)S-0.9SP-\beta_0(1+0.7\cos(2\pi t))SI\\
I' = \beta_0(1+0.7\cos(2\pi t))SI-0.7(1+0.7\cos(\pi+2\pi t))PI-0.1I\\
P' = (0.2-0.6P)P+0.09{SP}+0.63(1+0.7\cos(\pi+2\pi t))PI
\end{cases}.
\end{equation*}
When $\beta_0=0.1$ we obtain $\cR^u\approx -0.217<0$ ($\cR^u_{per}\approx 0.35<1$) and we conclude that we have extinction (figure~\ref{fig_exe2_extincao}). When $\beta_0=0.8$ we obtain  $\cR^\ell\approx 0.167>0$ ($\cR^\ell_{per}\approx 1.483>1$) and we conclude that the infectives are uniform strong persistent (figure~\ref{fig_exe2_persistencia}).

We considered the following initial conditions at $t=0$: $(S_0,I_0,P_0)=(1,0.5,0.1)$, $(S_0,I_0,P_0)=(0.1,0.2,1)$ and $(S_0,I_0,P_0)=(0.5,0.5,0.5)$. 

\begin{figure}[H]
  \begin{minipage}[b]{.32\linewidth}
    \includegraphics[width=\linewidth]{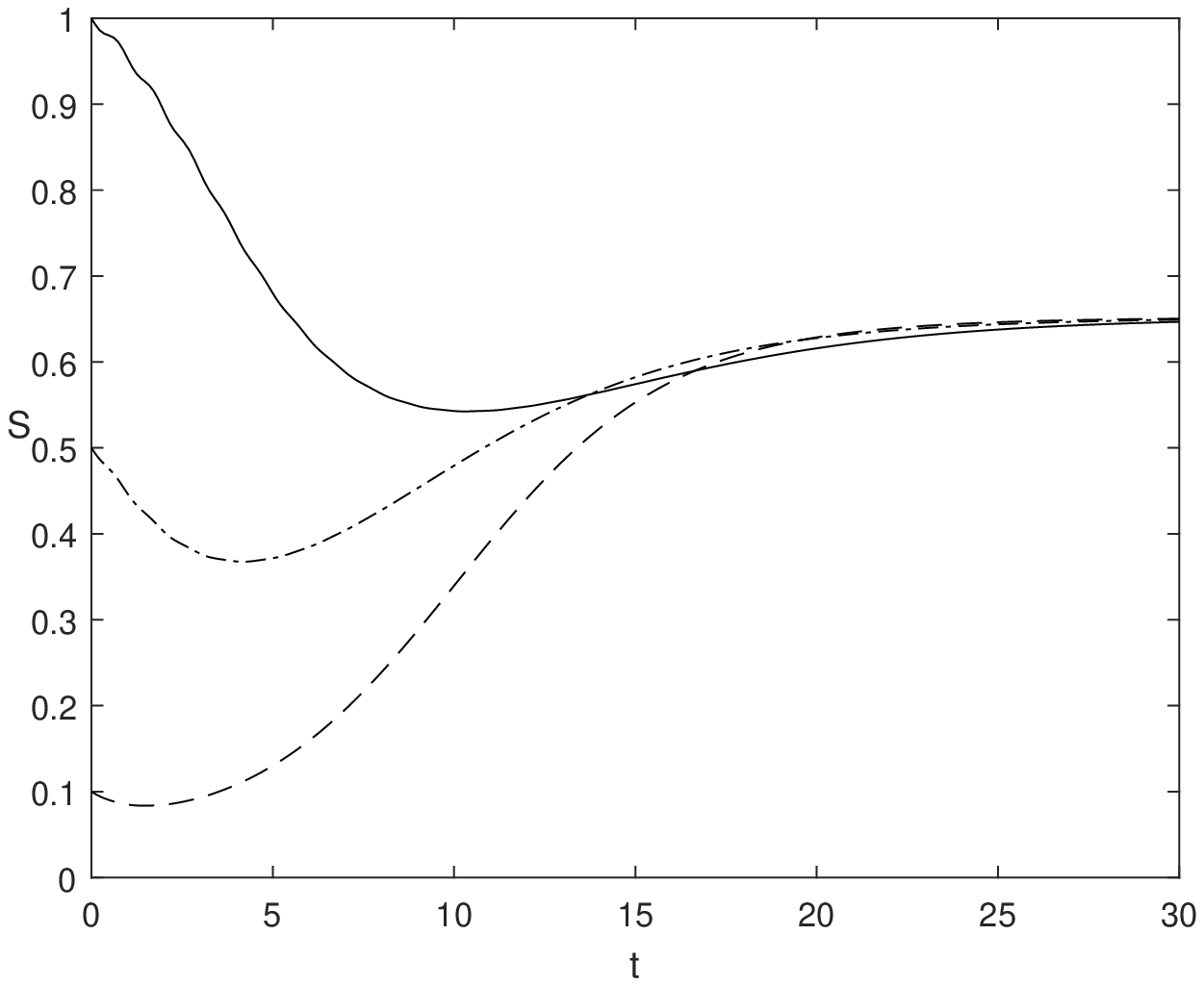}
  \end{minipage}
  \begin{minipage}[b]{.32\linewidth}
        \includegraphics[width=\linewidth]{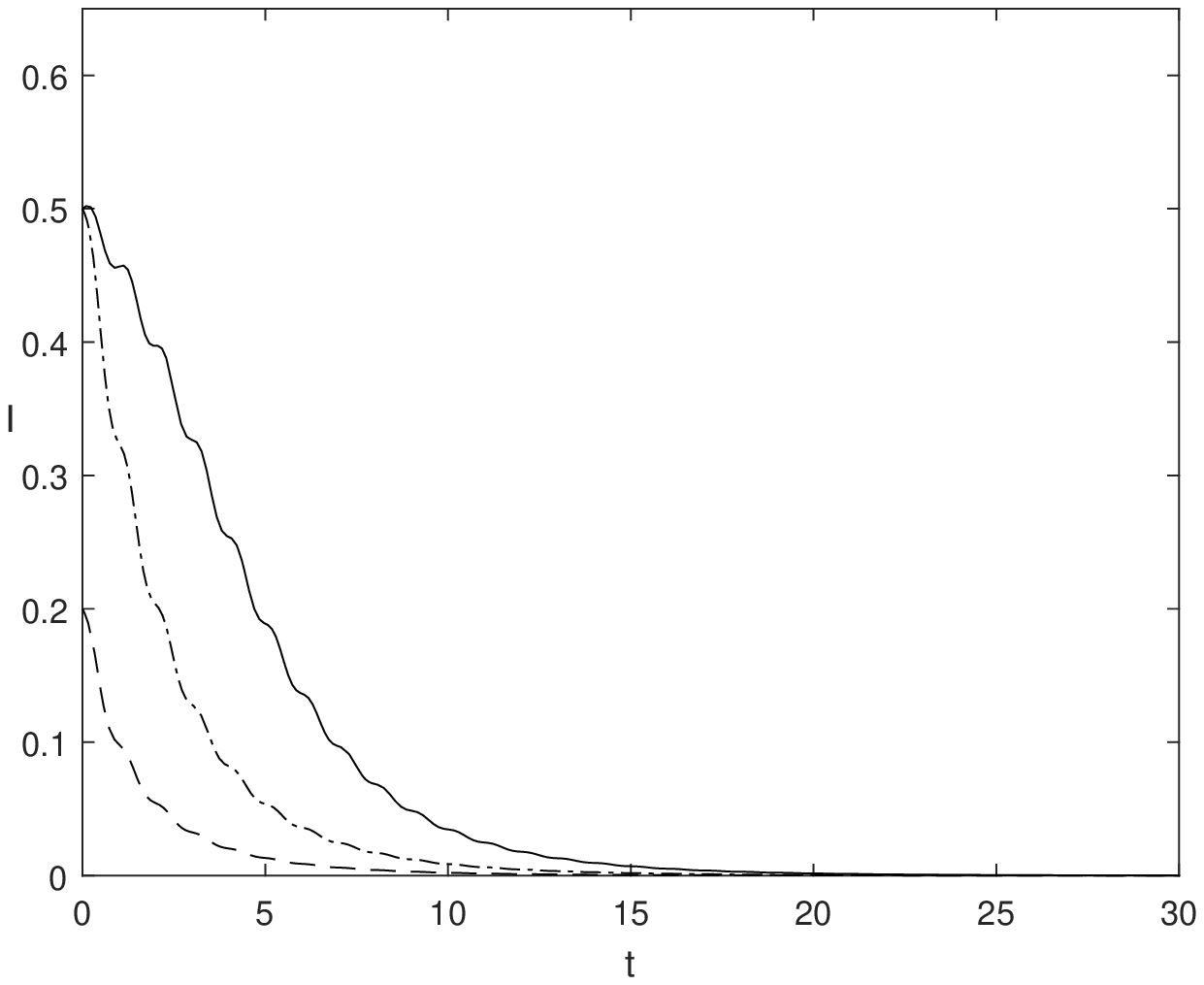}
  \end{minipage}
  \begin{minipage}[b]{.32\linewidth}
        \includegraphics[width=\linewidth]{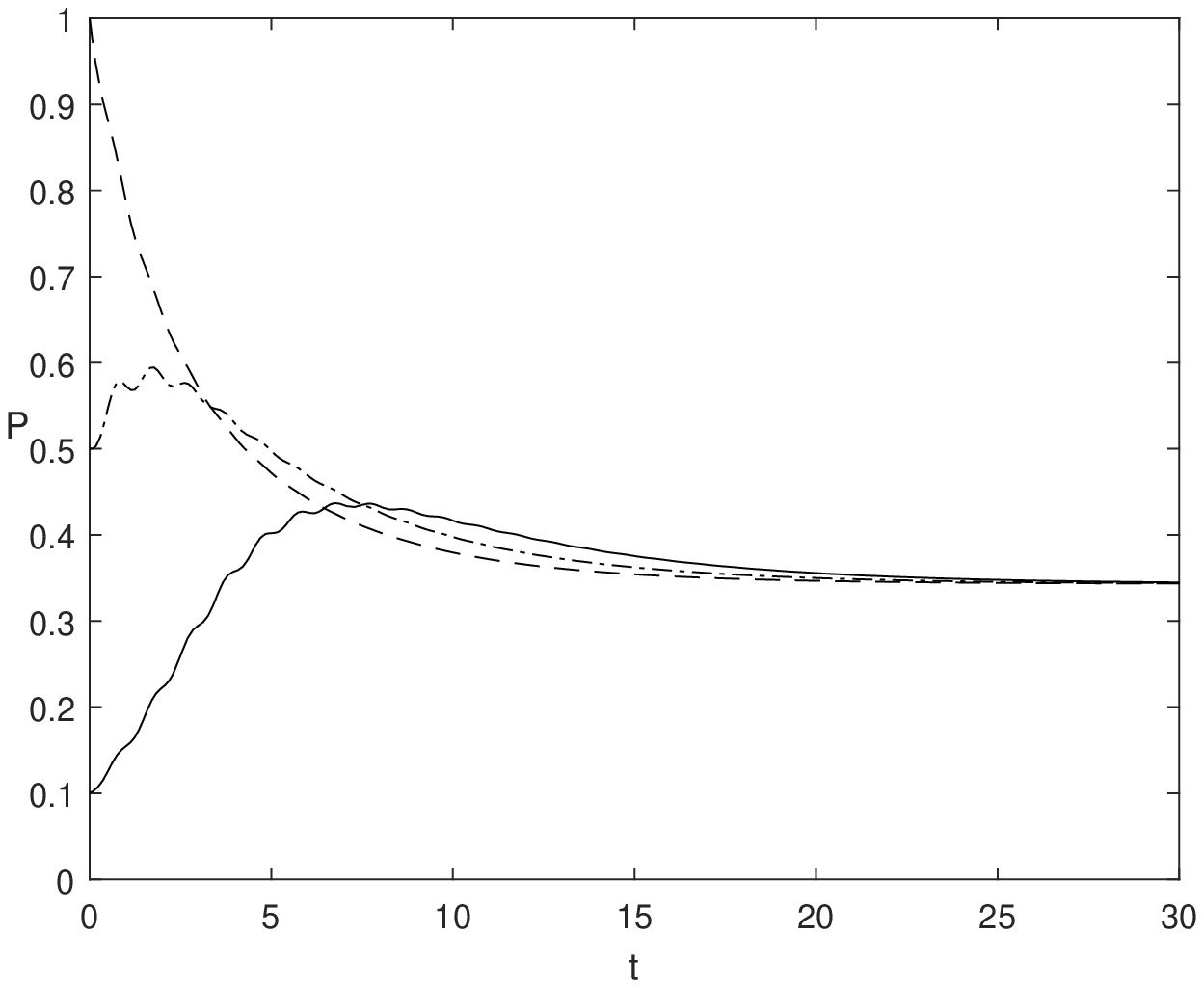}
  \end{minipage}
    \caption{Extinction: $\beta_0=0.1$.}
      \label{fig_exe2_extincao}
\end{figure}

\begin{figure}[H]
  \begin{minipage}[b]{.32\linewidth}
    \includegraphics[width=\linewidth]{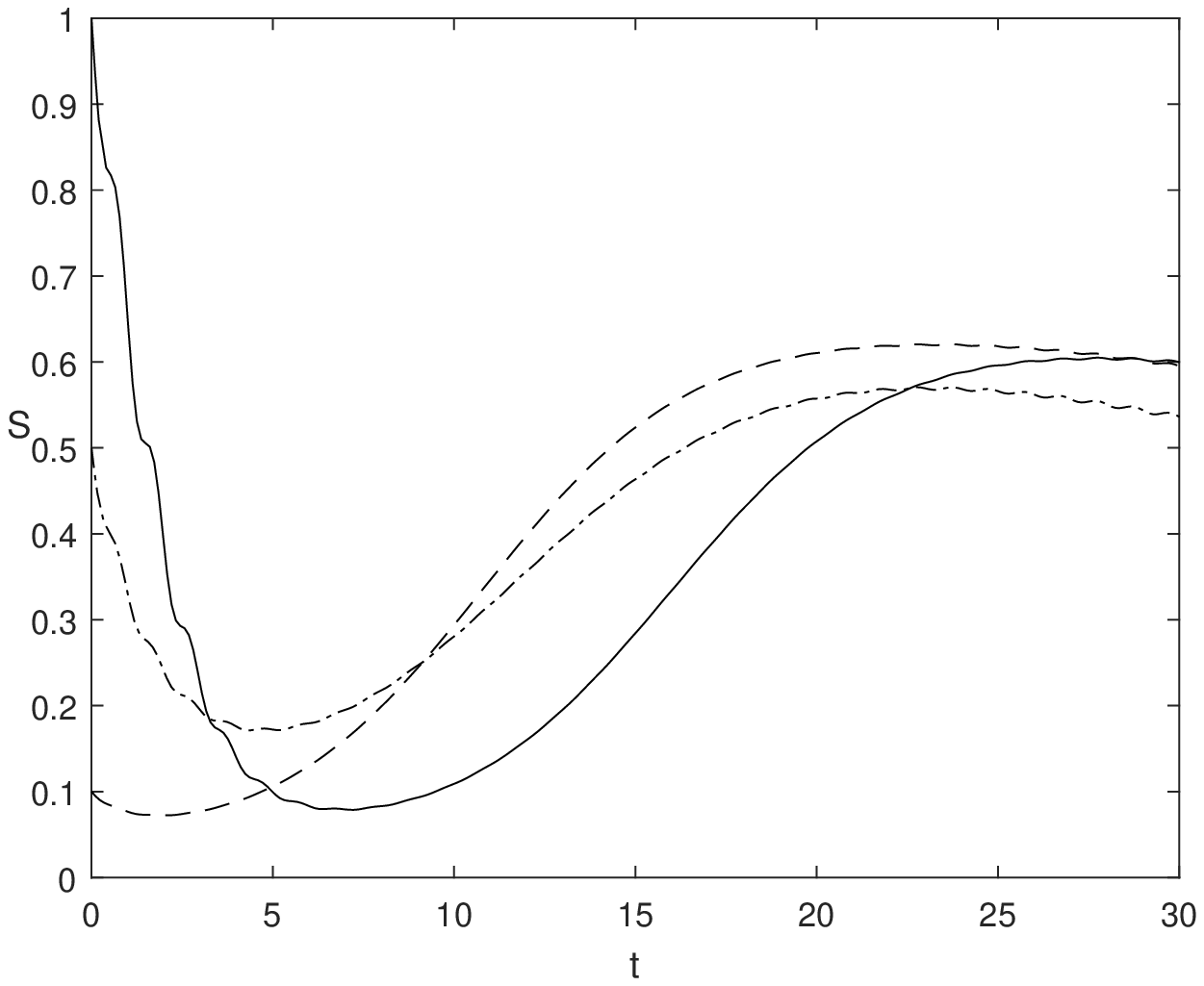}
  \end{minipage}
  \begin{minipage}[b]{.32\linewidth}
        \includegraphics[width=\linewidth]{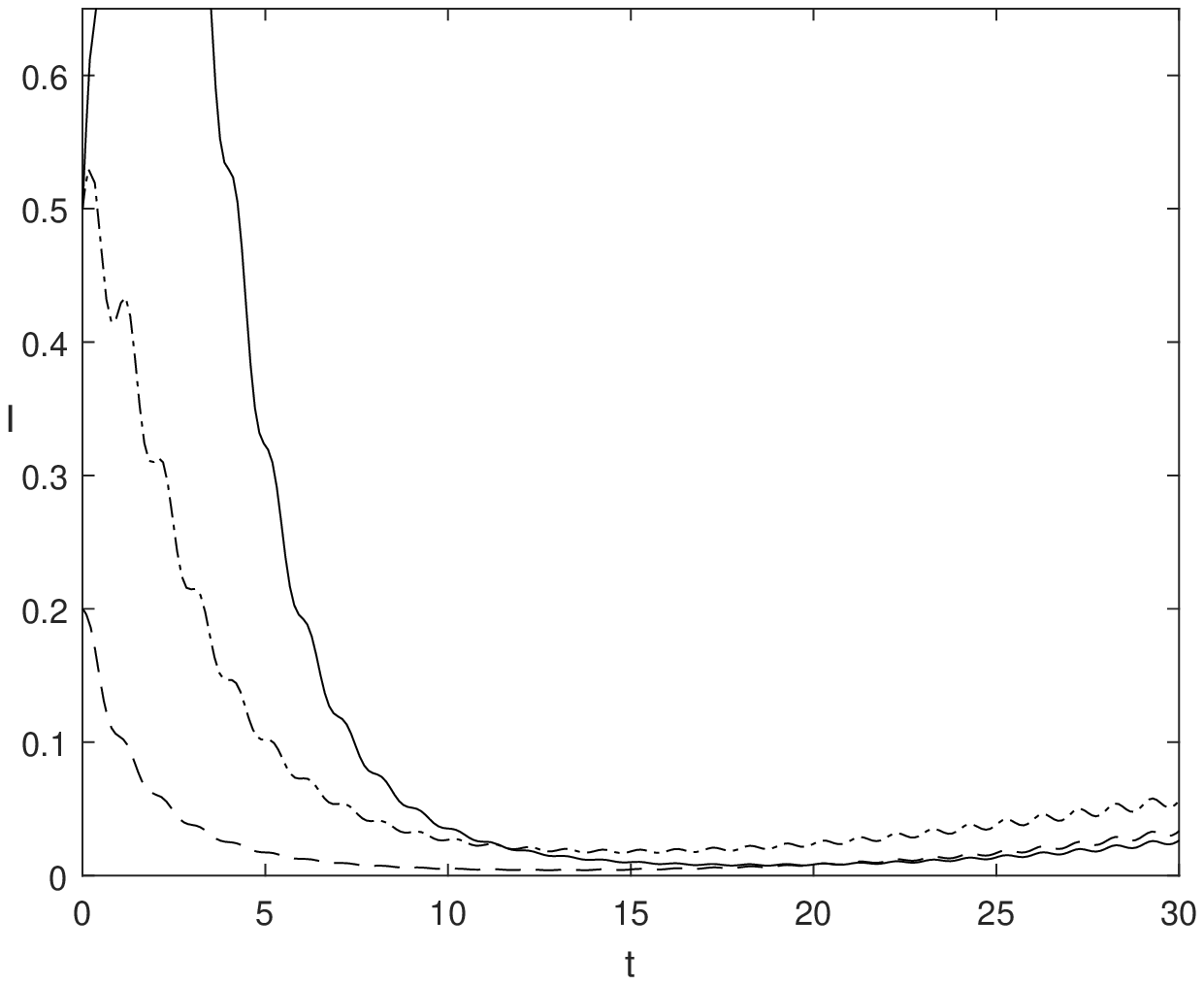}
  \end{minipage}
  \begin{minipage}[b]{.32\linewidth}
        \includegraphics[width=\linewidth]{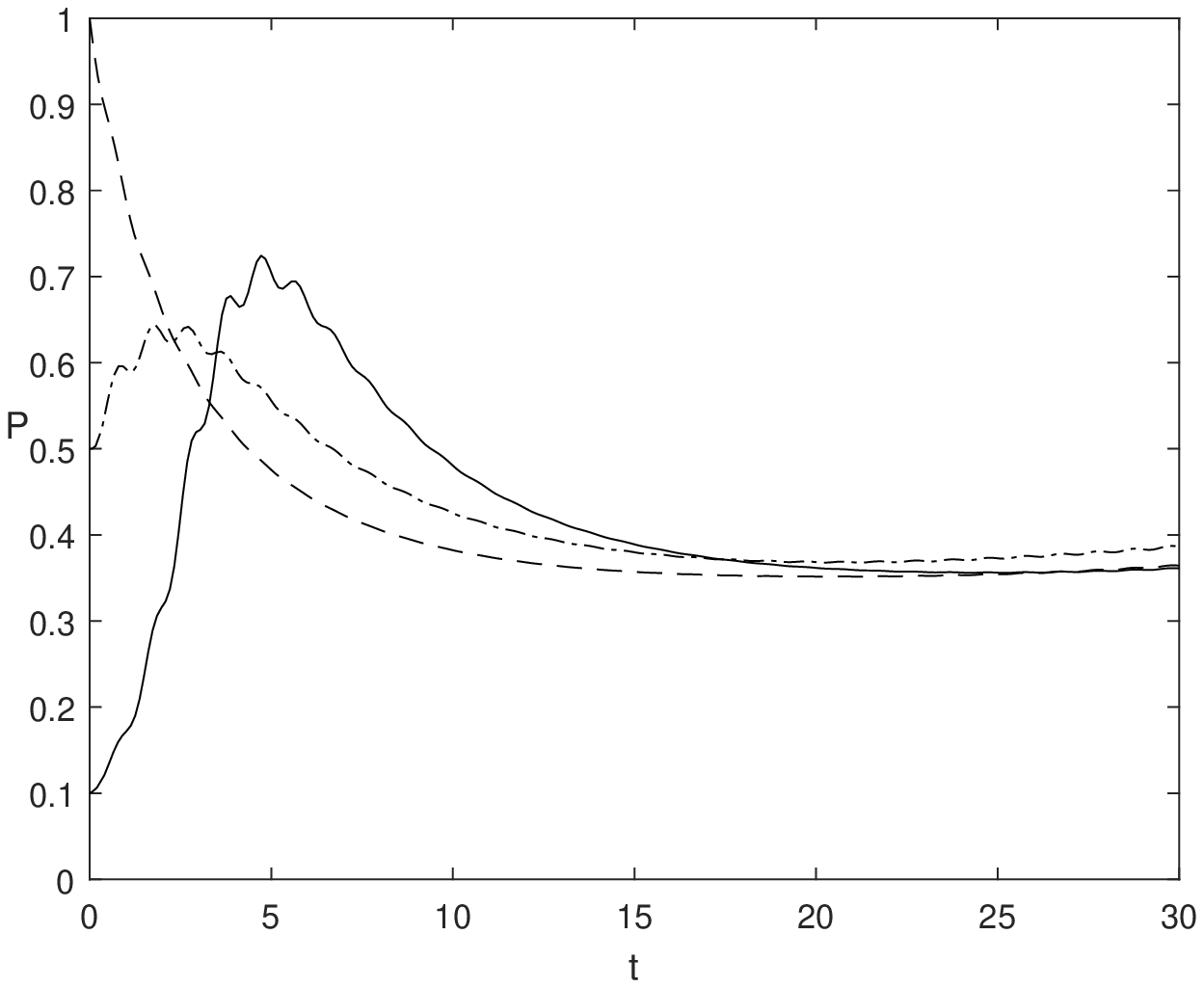}
  \end{minipage}
    \caption{Uniform strong persistence: $\beta_0=0.8$.}
      \label{fig_exe2_persistencia}
\end{figure}

\subsection{{Models with Gause-type uninfected subsystem}}
A model with Michaelis-Menten (or Holling-type I) functional response of predator to infected prey and a Holling-type II functional response of predator to susceptible prey is now considered. We consider the vital dynamics of uninfected prey as
	\[
	G(t,S)=k(t,S)S,
	\]
where $k\colon\R\times[0,+\infty[\to\R$ is continuous, $T$-periodic $(T > 0)$ in the $t$-variable, continuously differentiable in $S$ and satisfying the following conditions:
	\begin{equation}\label{eq:k(t,S)}
	\begin{cases}
	k(t,S)\,\text{ is bounded from above and }	k(t, 0) > 0;\\
	\text{ for every } t \text{ there exists } S_1(t) > 0 \text{ such that } k(t,S_1(t)) = 0;
	\\	
	\dfrac{\partial k}{\partial S}< 0 \text{ for every } S \geq 0.
	\end{cases}
	\end{equation}
	This type of vital dynamics was considered in \cite{Garrione-Rebelo-NARWA-2016} (see condition (k) in \S3 in that reference).
We notice that the general assumptions \eqref{eq:k(t,S)} are satisfied by a logistic growth of the prey population (with $S_1(t)$ equal to a constant). 
We also consider $a(t)=a>0$; $f(S,I,P)=S/(m+S+I)$, with $m>0$; $g(S,I,P)=P$; $h(t,P)=b-rP$, with $b, 
r>0$, $\gamma(t)=\gamma>0$. Moreover, we assume  $\beta(t), \eta(t), c(t)$ and $\theta(t)$ to be bounded, nonnegative and continuous real valued functions. Henceforth we are considering the system
\begin{equation}\label{eq:Periodic-second-scenario}
\begin{cases}
S'={k(t,S)S}- a \frac{SP}{m+S+I}-\beta(t)SI\\
I'=\beta(t)SI-\eta(t)PI-c(t)I\\
P'={(b-rP)P}+ \gamma a \frac{SP}{m+S+I}+\theta(t)\eta(t)PI
\end{cases},
\end{equation}
For model~\eqref{eq:Periodic-second-scenario}, conditions S\ref{cond-1}) and~S\ref{cond-3}) follow  immediate from hypothesis. The following lemma ensures condition~S\ref{cond-12}).

\begin{lemma}
There is a bounded region that contains the $\omega$-limit of all orbits of~\eqref{eq:Periodic-second-scenario}.
\end{lemma}
\begin{proof}
By the first equation in~\eqref{eq:Periodic-second-scenario} we have 
$
S' \le k(t,S)S$. Notice that from~\eqref{eq:k(t,S)} we have $k(t,S) > 0$ for $S < S_1(t)$ and $k(t,S) < 0$ for $S > S_1(t)$. Thus $[0,\max_{[0,T]}S_1(t)]$ is an attractor and so
\begin{equation*}
\limsup_{t\to+\infty}S(t)\le \max_{[0,T]}S_1(t):=S_1^u.
\end{equation*}
Additionally, writing $k^u=\sup_{(t,S)}k(t,S)$, we have
\begin{equation*}
\begin{split}
(S+I)'
& \leq k(t,S)S-c(t)I\\
& \leq k^uS-c^\ell I\\
& \leq (k^u + c^\ell)(S_1^u+\delta)-c^\ell(S+I),
\end{split}
\end{equation*}
for some $\delta>0$ and all sufficiently large $t$, which lead us to
\begin{equation*}
\begin{split}
(S+I)(t) \leq \frac{(k^u+c^\ell)(S_1^u+\delta)}{c^\ell}=:C.\\
\end{split}
\end{equation*}
Finally, since for sufficiently large $t$
\[
P'\leq (b-rP+\gamma a +\theta^u\eta^uC)P,
\]
we have
	\[
	\limsup_{t\to+\infty} P(t) \leq\frac1r({b+\gamma a +\theta^u\eta^uC})=:D.
	\]
Thus the region
$$\{(S,I,P) \in \R^3: 0 \le S+I \le C \ \text{and} \ 0 \le P \le D\}$$
contains the $\omega$-limit of any orbit. 
\end{proof}

Notice that by \cite[Lemma 3.1]{Garrione-Rebelo-NARWA-2016}, equations $s'=k(t,s)s$ and $y'=(b-ry)y$, with positive initial condition, have a unique $T$-periodic solution which is bounded and globally asymptotically stable, hence globally attractive, on $]0,+\infty[$. This ensures conditions S\ref{cond-9}) and S\ref{cond-10}). For condition~S\ref{cond-11}) we consider $G_{1,\eps}(t,x)=G_{2,\eps}(t,x)=k(t,x)x$, $h_{1,\eps}(t,z)=h_{2,\eps}(t,z)=(r-bz)z$, $v(\eps)=\eps$ and $\rho(t)=1$. Clearly, conditions S\ref{cond-61}.1) to S\ref{cond-63}.3) hold. The auxiliary subsystem~\eqref{eq:auxiliary-system-SP-2} becomes
\begin{equation}\label{eq:auxiliary-system-SP-2-exemplo-3}
\begin{cases}
x'=k(t,x)x\\
z'=(b-rz)z+\gamma a \frac{xz}{m+x}+\eps z.
\end{cases}
\end{equation}
By \cite[Lemma 3.1]{Garrione-Rebelo-NARWA-2016} there exists a $T$-periodic globally asymptotic stable solution $x_{2,\eps}^*(t)$ of the first equation of~\eqref{eq:auxiliary-system-SP-2-exemplo-3}. Considering this solution in the second equation we get
	\[
	z'=\left(b+\gamma a \frac{x_{2,\eps}^*(t)}{m+x_{2,\eps}^*(t)}+\eps-rz\right) z,
	\]
which, again from \cite[Lemma 3.1]{Garrione-Rebelo-NARWA-2016}, has a  $T$-periodic globally asymptotic stable solution $z_{2,\eps}^*(t)$.
Using this solutions $(x_{2,\eps}^*(t),z_{2,\eps}(t)^*)$, and writing
	\[\tilde k(t,x)=k(t,x)-  \frac{az_{2,\eps}^*(t)}{m+x}-\eps
	\]
the auxiliary subsystem~\eqref{eq:auxiliary-system-SP-1} becomes
\begin{equation}\label{eq:auxiliary-system-SP-1-exemplo-3}
\begin{cases}
x'=\tilde k(t,x)x\\
z'=(b-rz)z+\gamma a \frac{xz}{m+x+\eps}.
\end{cases}
\end{equation}
It is straightforward to verify that $\tilde k(t,S)$ is bounded from above. Moreover, if
\begin{equation}\label{eq:exemplo-3-k}
k(t,0)-\frac{az_{2,\eps}^*(t)}{m}-\eps>0
\end{equation}
for all sufficiently small $\eps>0$, then $\tilde k(t,0)>0$. Notice that we may find a bound for $z_{2,\eps}^*(t)$ independent of $\eps$.
Moreover, if 
\begin{equation}\label{eq:exemplo-3-k-partial}
\sup_{(t,x)}\left\{ \frac{\partial k}{\partial x}(t,x)+\frac{az_{2,\eps}^*(t)}{(m+x)^2}\right\}<0
\end{equation}
we have
\begin{equation*}
	\frac{\partial\tilde k}{\partial x}(t,x)=		\frac{\partial k}{\partial x}(t,x)+\frac{az_{2,\eps}^*(t)}{(m+x)^2}\leq\zeta<0.
	\end{equation*}
If~\eqref{eq:exemplo-3-k} and \eqref{eq:exemplo-3-k-partial} hold we have $\frac{\partial\tilde k}{\partial x}(t,x)\leq\zeta<0$ and  $\tilde k(t,0)>0$. Since $\lim_{x\to+\infty}\tilde k(t,x)=-\infty$ we conclude that for all $t$ there exists $X_1(t)$ such that $\tilde k(t,X_1(t))=0$.
Notice that having simultaneously~\eqref{eq:exemplo-3-k} and \eqref{eq:exemplo-3-k-partial} can be achieved, for instance, if $a$ is small enough.
Therefore, if~\eqref{eq:exemplo-3-k} and \eqref{eq:exemplo-3-k-partial} hold and $\eps>0$ is sufficiently small we are in conditions to apply \cite[Lemma 3.1]{Garrione-Rebelo-NARWA-2016} to conclude that
the first equation in~\eqref{eq:auxiliary-system-SP-1-exemplo-3} has a $T$-periodic solution $x_{1,\eps}^*(t)$ which is globally asymptotically stable. Using this solution in the second equation gives
	\[
	z'=(b-rz)z+\gamma a \frac{x_{1,\eps}^*(t)}{m+x_{1,\eps}^*(t)+{\eps}}z=\left(b+\gamma a \frac{x_{1,\eps}^*(t)}{m+x_{1,\eps}^*(t)+{\eps}}-rz\right)z,
	\]
which, proceeding as before, has a $T$-periodic solution $z_{1,\eps}^*(t)$ that is globally asymptotically.

We have showed that conditions~S\ref{eq:subsyst-1}.4) and~S\ref{cond-65}.5) also hold, so we may
 apply Theorem~\ref{teo:main-extinction} and Theorem~\ref{teo:main-persistence} to conclude that if $\cR^u<1$ then the infectives in model~\eqref{eq:Periodic-second-scenario} go to extinction, and if $\cR^\ell>1$ then the infectives are uniform strong persistent.

Let us point out that in the particular case $k(t,x)=\Lambda-\mu x$, for some constants $\Lambda,\mu>0$, we have the following globally asymptotically stable solution for~\eqref{eq:auxiliary-system-SP-2-exemplo-3}
\[
\begin{cases}
x_{2,\eps}^*=\Lambda/\mu\\
z_{2,\eps}^*=\left(b+\gamma a\frac{1}{m\mu+1}+\eps\right)/r
\end{cases}
\]
and in this situation conditions~\eqref{eq:exemplo-3-k} and \eqref{eq:exemplo-3-k-partial} holds if
	\[\frac{a}{mr}\left(b+\frac{\gamma a}{m\mu+1}\right)<\min\{\Lambda,m\mu\}
		\]
and $\eps$ is sufficiently small.

\,

To do some simulation, in this scenario we assumed that $G(t,S)=(0.7-0.6S)S$; $a=0.9$; $\beta(t)=\beta_0(1+0.7\cos(2\pi t))$; $\eta(t)=0.7(1+0.7\cos(\pi+2\pi t))$; $c(t)=0.1$; $b=0.2$; $r=0.6$; $m=2$; $\gamma=0.8$; $\theta(t)=0.9$. We obtain the model:
\begin{equation*}
\begin{cases}
S' =(0.7-0.6S)S-0.9 \, \frac{SP}{2+S+I}-\beta_0(1+0.7\cos(2\pi t))SI\\
I' = \beta_0(1+0.7\cos(2\pi t))SI-0.7(1+0.7\cos(\pi+2\pi t))PI-0.1I\\
P' = (0.2-0.6P)P+0.9 \, \frac{SP}{2+S+I}+0.6(1+0.7\cos(\pi+2\pi t))PI
\end{cases}.
\end{equation*}
When $\beta_0=0.1$ we obtain $\cR^u\approx -0.217<0$ and we conclude that we have extinction (figure~\ref{fig_exe3_extincao}). When $\beta_0=0.9$ we obtain  $\cR^\ell\approx 0.757>0$ and we conclude that the infectives are uniform strong persistent (figure~\ref{fig_exe3_persistencia}).

We considered the following initial conditions at $t=0$: $(S_0,I_0,P_0)=(1,0.5,0.1)$, $(S_0,I_0,P_0)=(0.1,0.2,1)$ and $(S_0,I_0,P_0)=(0.5,0.5,0.5)$. 

\begin{figure}[H]
  \begin{minipage}[b]{.32\linewidth}
    \includegraphics[width=\linewidth]{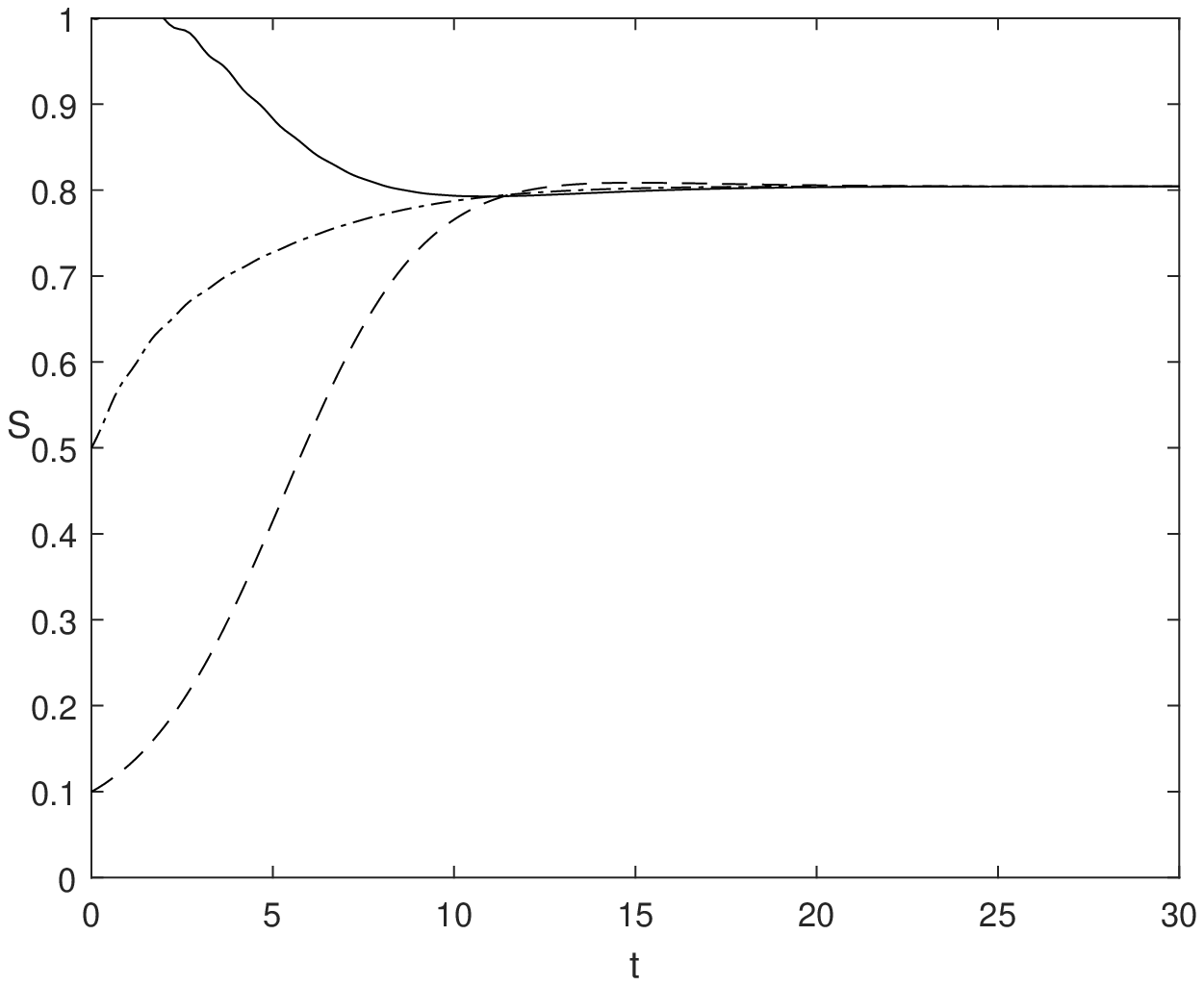}
  \end{minipage}
  \begin{minipage}[b]{.32\linewidth}
        \includegraphics[width=\linewidth]{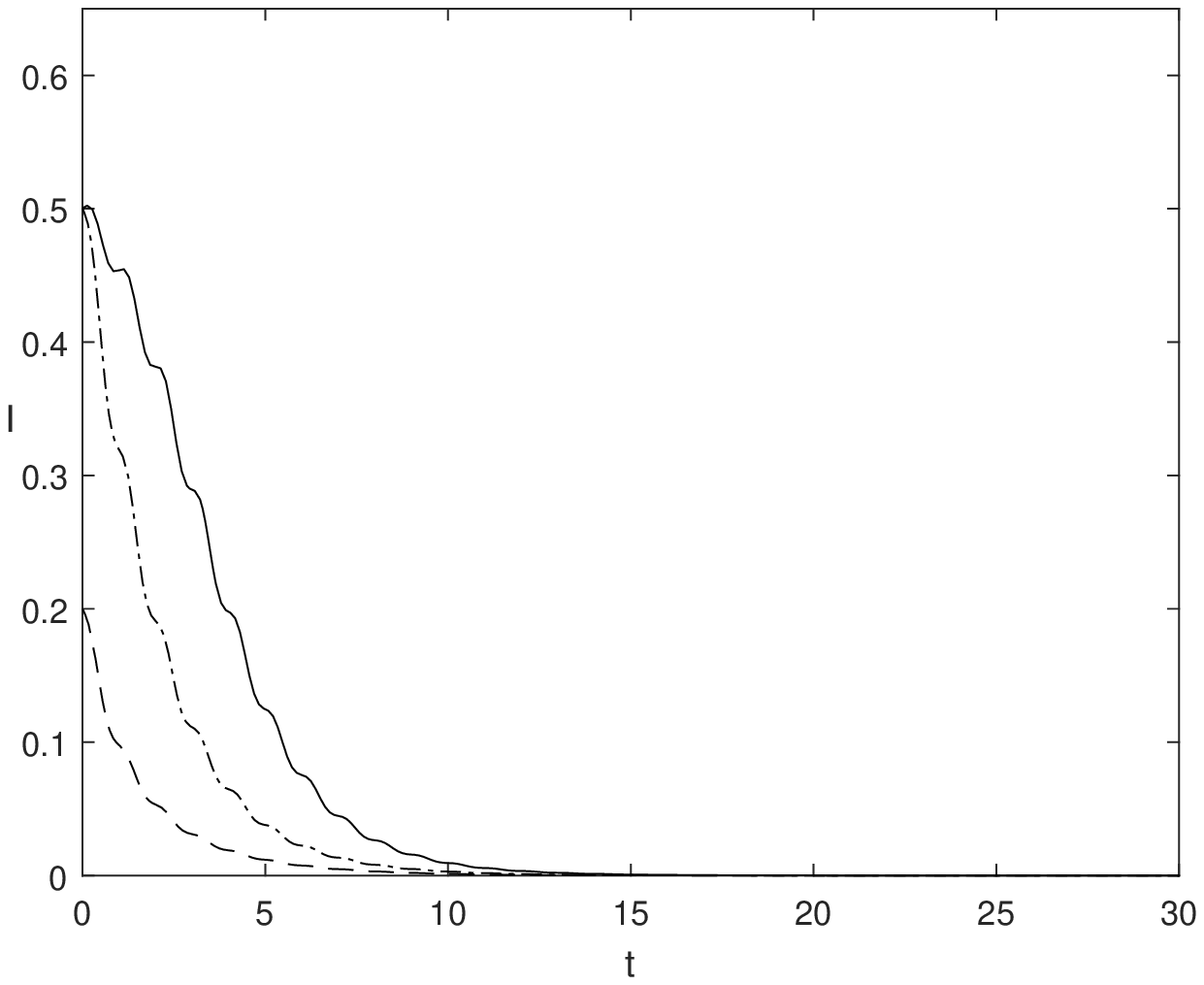}
  \end{minipage}
  \begin{minipage}[b]{.32\linewidth}
        \includegraphics[width=\linewidth]{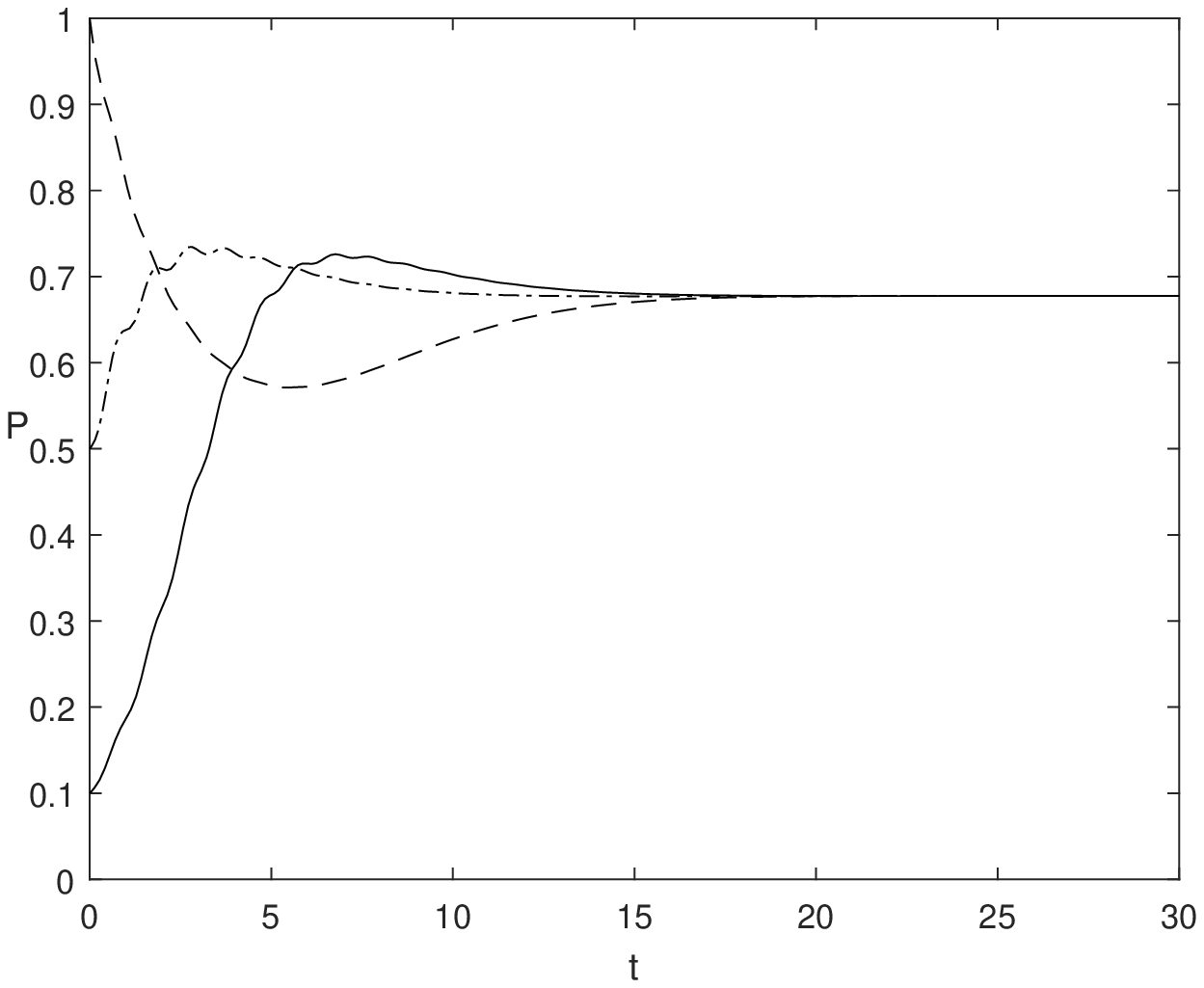}
  \end{minipage}
    \caption{Extinction: $\beta_0=0.1$.}
      \label{fig_exe3_extincao}
\end{figure}

\begin{figure}[H]
  \begin{minipage}[b]{.32\linewidth}
    \includegraphics[width=\linewidth]{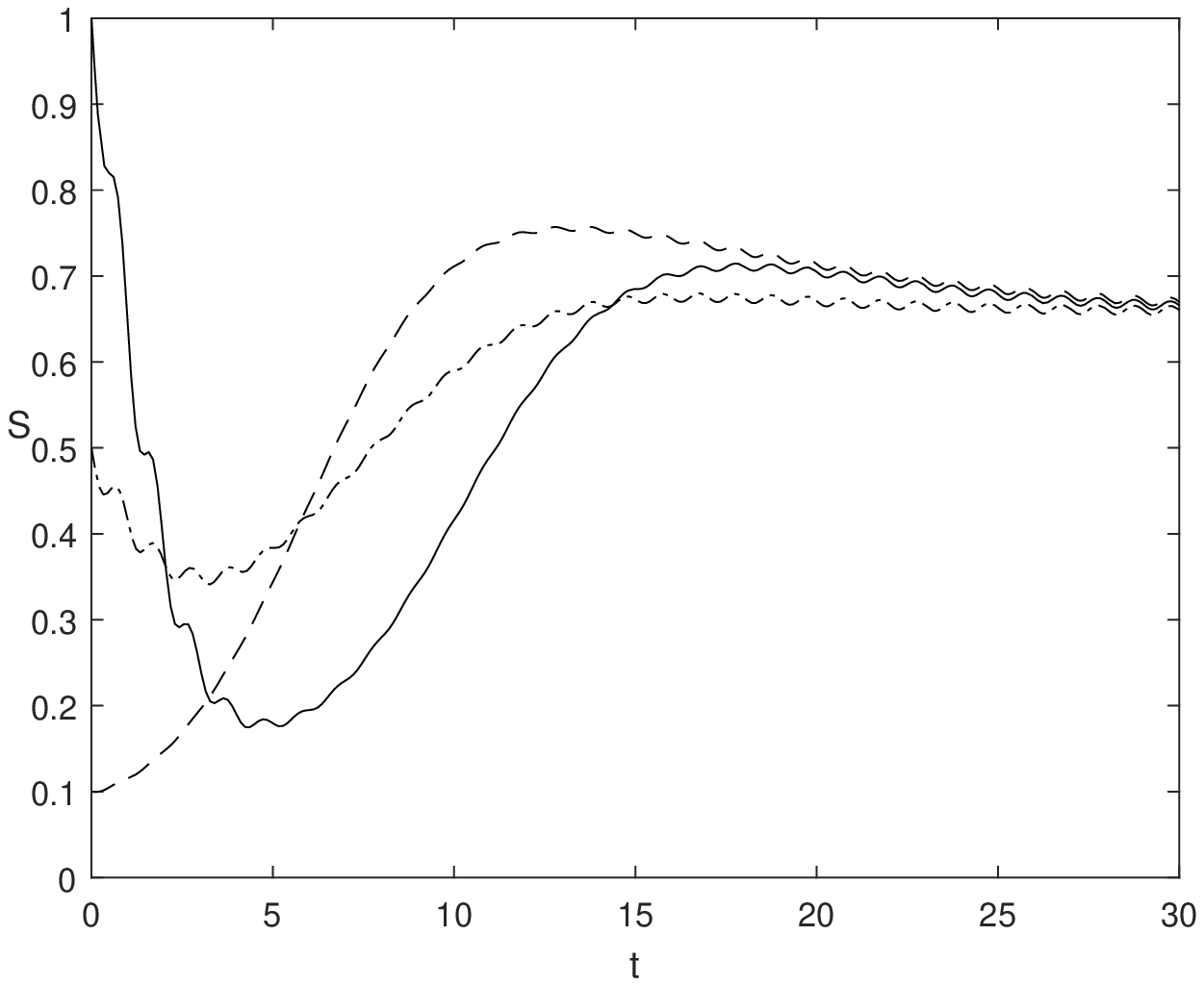}
  \end{minipage}
  \begin{minipage}[b]{.32\linewidth}
        \includegraphics[width=\linewidth]{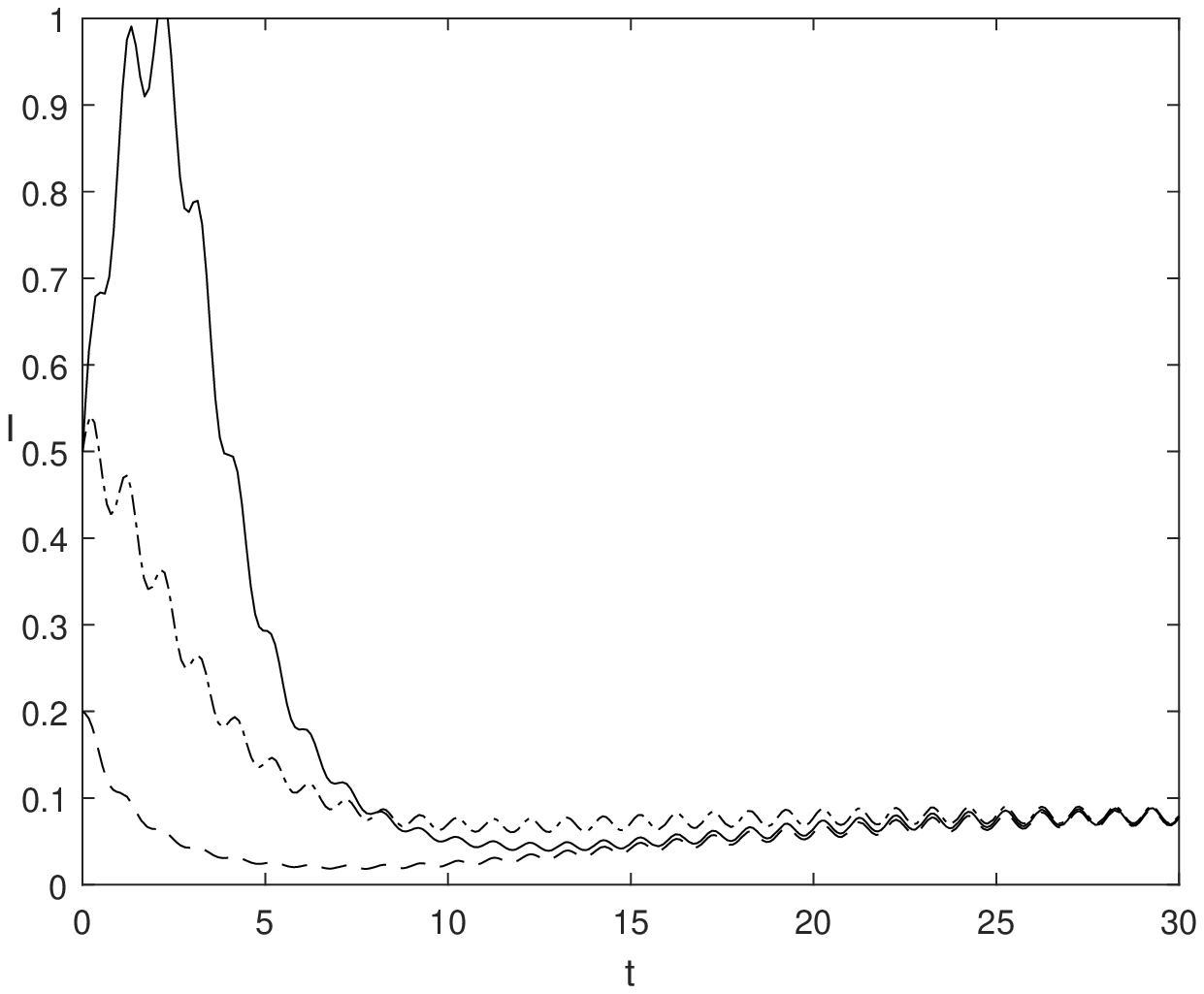}
  \end{minipage}
  \begin{minipage}[b]{.32\linewidth}
        \includegraphics[width=\linewidth]{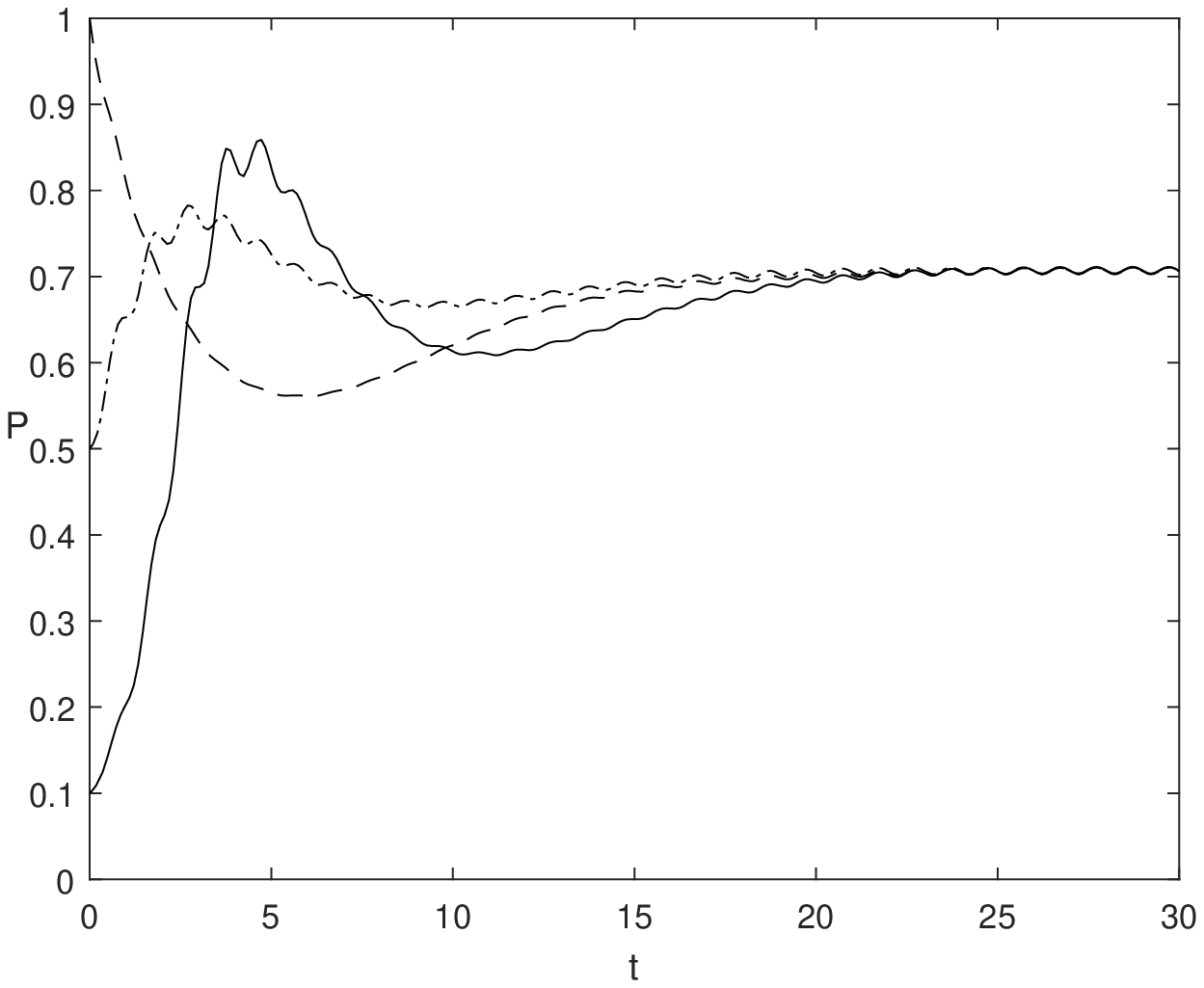}
  \end{minipage}
    \caption{Uniform strong persistence: $\beta_0=0.9$.}
      \label{fig_exe3_persistencia}
\end{figure}

\subsection{Models with ratio-dependent uninfected subsystem}

{The functional response of predator to prey in the uninfected subsystem in the next example is ratio-dependent. Ratio-dependent functional responses were considered to overcome some paradoxes identified in Gause-type systems (see~\cite{Hsu-Hwang-Kuang-MB-2001} and the references therein).}

We consider now the vital dynamics of uninfected prey as
	\[
	G(t,S)=\Lambda - \mu S,
	\]
for constants $\Lambda, \mu >0$, and $a(t)=a>0$, $f(S,I,P)=S/(Pm+S+I)$, with $m>0$, $g(S,I,P)=P$; $h(t,P)=b-rP$, with $b, 
r>0$, $\gamma(t)=\gamma>0$, and $\beta(t), \eta(t), c(t)$ and $\theta(t)$ to be  bounded, nonnegative and continuous real valued functions. With this assumptions system~\eqref{eq:principal} becomes
\begin{equation}\label{eq:example-4}
\begin{cases}
S'={\Lambda-\mu S} - a \frac{SP}{mP+S+I}-\beta(t)SI\\
I'=\beta(t)SI-\eta(t)PI-c(t)I\\
P'={(b-rP)P}+ \gamma a \frac{SP}{mP+S+I}+\theta(t)\eta(t)PI
\end{cases}.
\end{equation}
For model~\eqref{eq:example-4}, conditions S\ref{cond-1}) and~S\ref{cond-3}) follow  immediate from hypothesis and the particular forms of the functions $g$ and $h$. Condition~S\ref{cond-12}) is an immediate corollary of the following.

\begin{lemma}
There is a bounded region that contains the $\omega$-limit of all orbits of~\eqref{eq:Periodic-second-scenario}.
\end{lemma}
\begin{proof}
By the first equation in~\eqref{eq:example-4} we have 
$
S' \le \Lambda-\mu S$ which implies
\begin{equation*}
\limsup_{t\to+\infty}S(t)\le \frac\Lambda\mu.
\end{equation*}
We have
\begin{equation*}
\begin{split}
(S+I)'
& \leq \Lambda-\mu S-c(t)I\\
& \leq \Lambda -\min\{\mu,c^\ell\}(S+I)\\
\end{split}
\end{equation*}
 which implies
\begin{equation*}
\begin{split}
(S+I)(t) \leq \frac{\Lambda}{\min\{\mu,c^\ell\}}+\delta=:C.\\
\end{split}
\end{equation*}
for some $\delta>0$ and all sufficiently large $t$. Finally, since for large $t$,
\[
P'\leq (b-rP+\gamma a +\theta^u\eta^uC)P
\]
we have
	\[
	\limsup_{t\to+\infty} P(t) \leq\frac1r({b+\gamma a +\theta^u\eta^uC})=:D.
	\]
Thus the region
$$\{(S,I,P) \in \R^3: 0 \le S+I \le C \ \text{and} \ 0 \le P \le D\}$$
contains the $\omega$-limit of any orbit. 
\end{proof}

Similarly to the previous example, each solution of $s'=\Lambda-\mu s$ and of $y'=(b-ry)y$, with positive initial condition is bounded, bounded away from zero, and globally asymptotic stable (hence globally attractive)  on $]0,+\infty[$, which ensures conditions S\ref{cond-9}) and S\ref{cond-10}).

For condition~S\ref{cond-11}) we consider now $G_{1,\eps}(t,x)=G_{2,\eps}(t,x)=\Lambda-\mu x$, $h_{1,\eps}(t,z)=h_{2,\eps}(t,z)=(r-bz)z$, $v(\eps)=\eps$ and $\rho(t)=1$. As before, conditions S\ref{cond-61}.1) to S\ref{cond-63}.3) are straightforward. The auxiliary subsystem~\eqref{eq:auxiliary-system-SP-2} becomes
\begin{equation}\label{eq:auxiliary-system-SP-2-exemplo-4}
\begin{cases}
x'=\Lambda-\mu x\\
z'=(b-rz)z+\gamma a \frac{xz}{mz+x}+\eps z.
\end{cases}
\end{equation}
The first equation of~\eqref{eq:auxiliary-system-SP-2-exemplo-4} has solution $x_{2,\eps}^*(t)=\Lambda/\mu$. Considering this solution in the second equation we get
	\[
	z'=\left(b+\gamma a \frac{\Lambda/\mu}{mz+\Lambda/\mu}+\eps-rz\right) z,
	\]
which has solution
\[
z_{2,\eps}^*(t)^*=\frac{A+\sqrt{A^2+4\Lambda rm\mu B}}{2rm\mu}:=z_{2,\eps},
\]
where
\[A=-\Lambda r + b m \mu + \eps m \mu\quad\text{ and }\quad B= b  + \eps  + a \gamma.\]
Proceeding as in previous example, from \cite[Lemma 3.1]{Garrione-Rebelo-NARWA-2016} we conclude that $(x_{2,\eps}^*(t),z_{2,\eps}^*(t))$ is a globally asymptotically stable solution on $\{(x,y)\in\R^2\colon x, y>0\}$. Using this solution $(x_{2,\eps}^*(t),z_{2,\eps}^*(t))$, the auxiliary subsystem~\eqref{eq:auxiliary-system-SP-1} becomes
\begin{equation*}
\begin{cases}
x'=\Lambda-\mu x-a z_{2,\eps}^*-\eps x\\
z'=(b-rz)z+\gamma a \frac{xz}{zm+x+\eps}.
\end{cases}
\end{equation*}
The first equation can be written as
\[x'=\Lambda-a z_{2,\eps}^*-(\mu+\eps) x,\]
which, if
\begin{equation}\label{eq:ex4-cond1}
\Lambda>a z_{2,\eps}^*
\end{equation}
it has solution
\[
x_{1,\eps}^*(t)=\frac{\Lambda-a z_{2,\eps}^*}{\mu+\eps}=:x_{1,\eps}^*.
\]
Using this solution in the second equation gives
	\[
	z'=\left(b+\gamma a \frac{x_{1,\eps}^*}{zm+x_{1,\eps}^*+\eps}-rz\right)z,
	\]
which in turn has solution
\[
z_{1,\eps}^*(t)=\frac{\tilde A + \sqrt{\tilde A + 4 m r \tilde B}}{2 m r}
\]
where
 $\tilde A=b m - \eps r - r x_{1,\eps}^*$  and $\tilde B=b\eps + b x_{1,\eps}^* + a \gamma x_{1,\eps}^*$.

\[
z_{1,\eps}^*(t)=\frac{\tilde A+\sqrt{\tilde A+4rm\tilde B}}{2rm},
\] 
where
 $\tilde A=mb-rx_{1,\eps}^*$  and $\tilde B=(b+\gamma a)x_{1,\eps}^*$.
Once more, proceeding as in the previous example, one can check that  by \cite[Lemma 3.1]{Garrione-Rebelo-NARWA-2016} we conclude that $(x_{1,\eps}^*(t),z_{1,\eps}^*(t))$ is a globally asymptotically stable solution on $\{(x,y)\in\R^2\colon x, y>0\}$.

We have showed that if~\eqref{eq:ex4-cond1} holds, which happens, for instance, for sufficiently small $a$, then conditions~S\ref{eq:subsyst-1}.4) and~S\ref{cond-65}.5) also hold, and we may
 apply Theorem~\ref{teo:main-extinction} and Theorem~\ref{teo:main-persistence} to conclude that if $\cR^u<1$ then the infectives in model~\eqref{eq:example-4} go to extinction, and if $\cR^\ell>1$ then the infectives are uniform strong persistent.\\

\,

To do some simulation, in this scenario we assumed that $G(t,S)=3-0.6S$; $a=0.9$; $\beta(t)=\beta_0(1+0.7\cos(2\pi t))$; $\eta(t)=0.7(1+0.7\cos(\pi+2\pi t))$; $c(t)=0.1$; $b=0.2$; $r=0.6$; $m=2$; $\gamma=0.8$; $\theta(t)=0.9$. We obtain the model:
\begin{equation*}
\begin{cases}
S' =3-0.6S-0.9 \, \frac{SP}{2P+S+I}-\beta_0(1+0.7\cos(2\pi t))SI\\
I' = \beta_0(1+0.7\cos(2\pi t))SI-0.7(1+0.7\cos(\pi+2\pi t))PI-0.1I\\
P' = (0.8-0.6P)P+0.9 \, \frac{SP}{2P+S+I}+0.6(1+0.7\cos(\pi+2\pi t))PI
\end{cases}.
\end{equation*}
When $\beta_0=0.01$ we obtain $\cR^u\approx -0.283<0$ and we conclude that we have extinction (figure~\ref{fig_exe4_extincao}). When $\beta_0=0.3$ we obtain  $\cR^\ell\approx 0.073>0$ and we conclude that the infectives are uniform strong persistent (figure~\ref{fig_exe4_persistencia}).

We considered the following initial conditions at $t=0$: $(S_0,I_0,P_0)=(1,0.5,0.1)$, $(S_0,I_0,P_0)=(0.1,0.2,1)$ and $(S_0,I_0,P_0)=(0.5,0.5,0.5)$. 

\begin{figure}[H]
  \begin{minipage}[b]{.32\linewidth}
    \includegraphics[width=\linewidth]{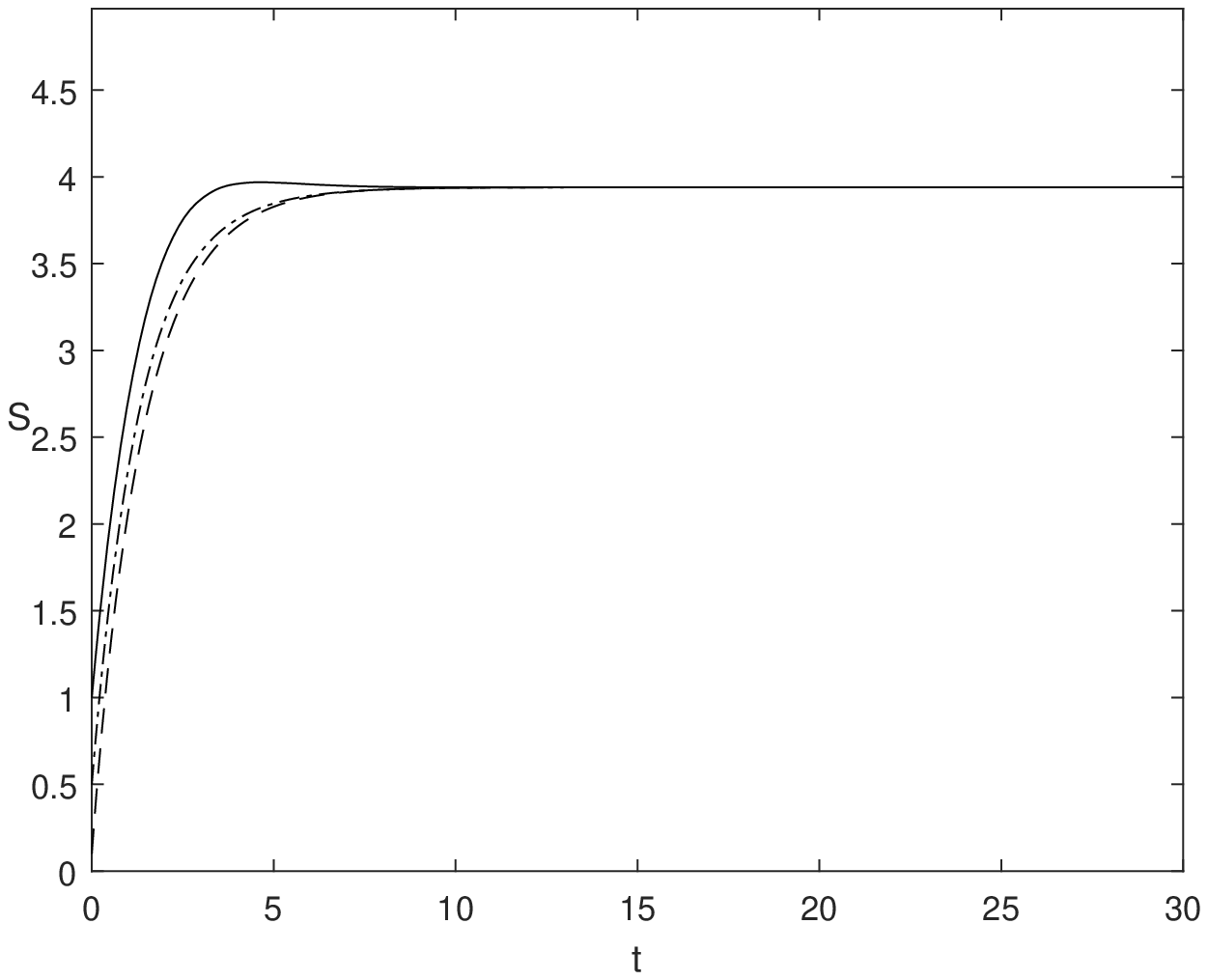}
  \end{minipage}
  \begin{minipage}[b]{.32\linewidth}
        \includegraphics[width=\linewidth]{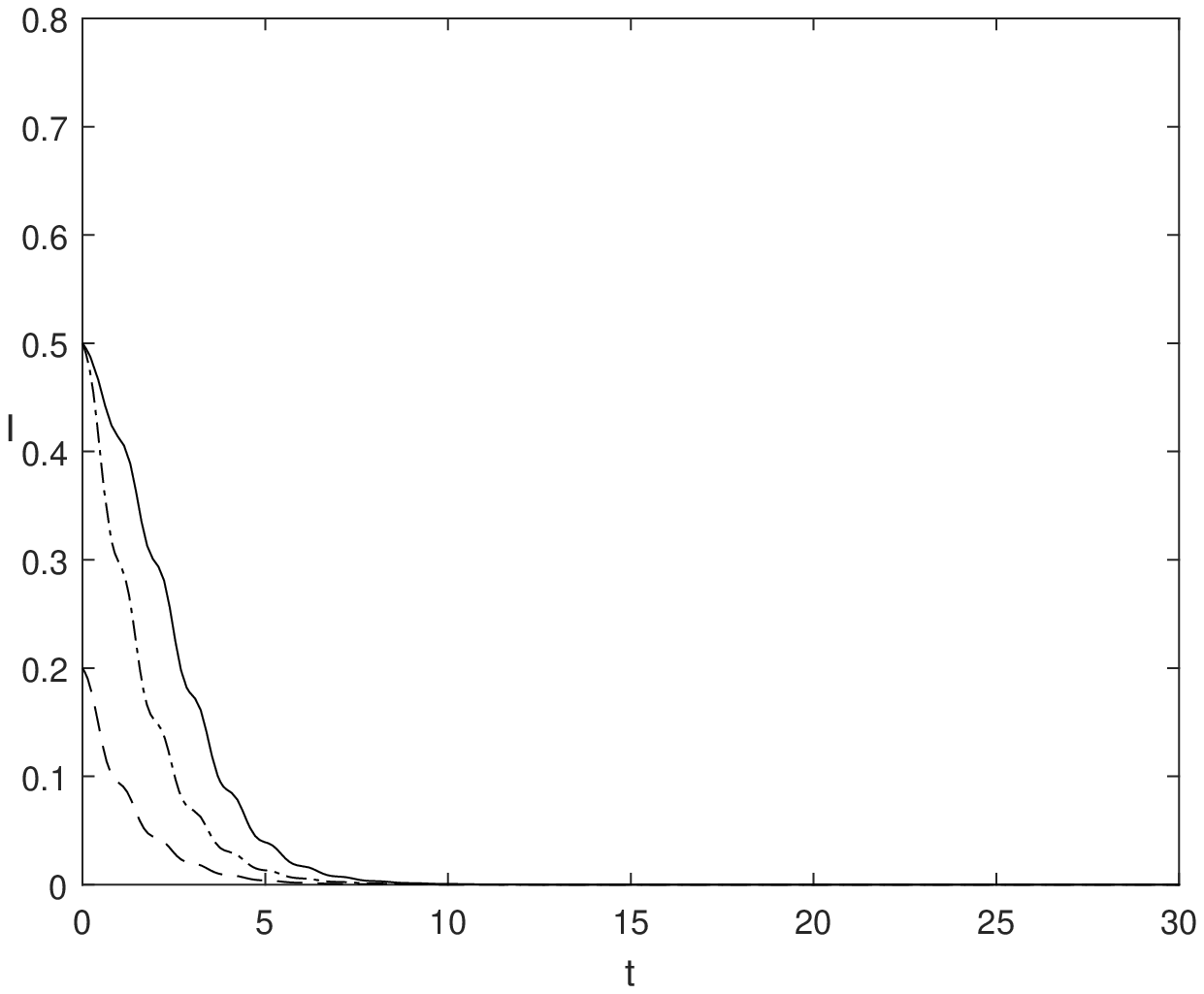}
  \end{minipage}
  \begin{minipage}[b]{.32\linewidth}
        \includegraphics[width=\linewidth]{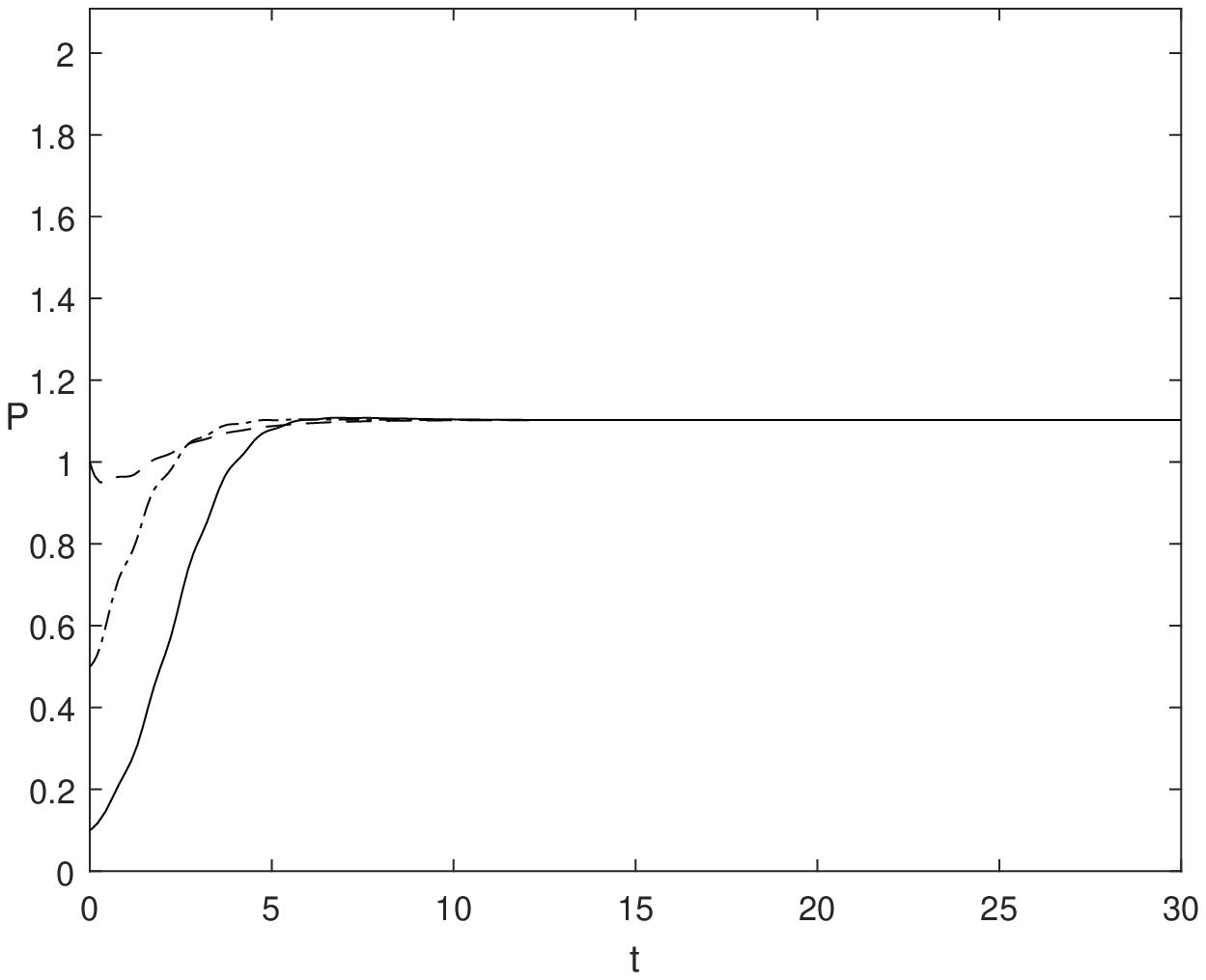}
  \end{minipage}
    \caption{Extinction: $\beta_0=0.01$.}
      \label{fig_exe4_extincao}
\end{figure}

\begin{figure}[H]
  \begin{minipage}[b]{.32\linewidth}
    \includegraphics[width=\linewidth]{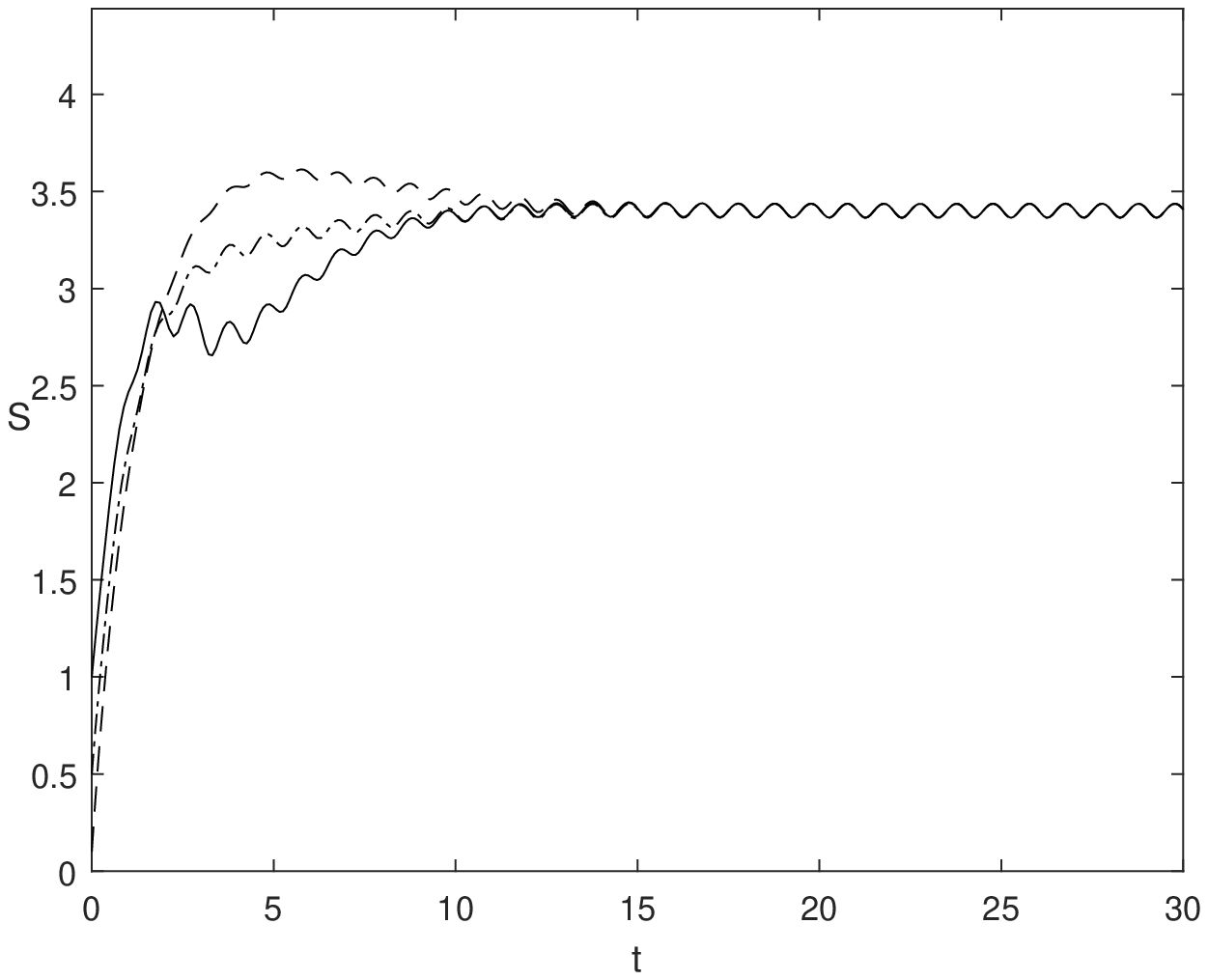}
  \end{minipage}
  \begin{minipage}[b]{.32\linewidth}
        \includegraphics[width=\linewidth]{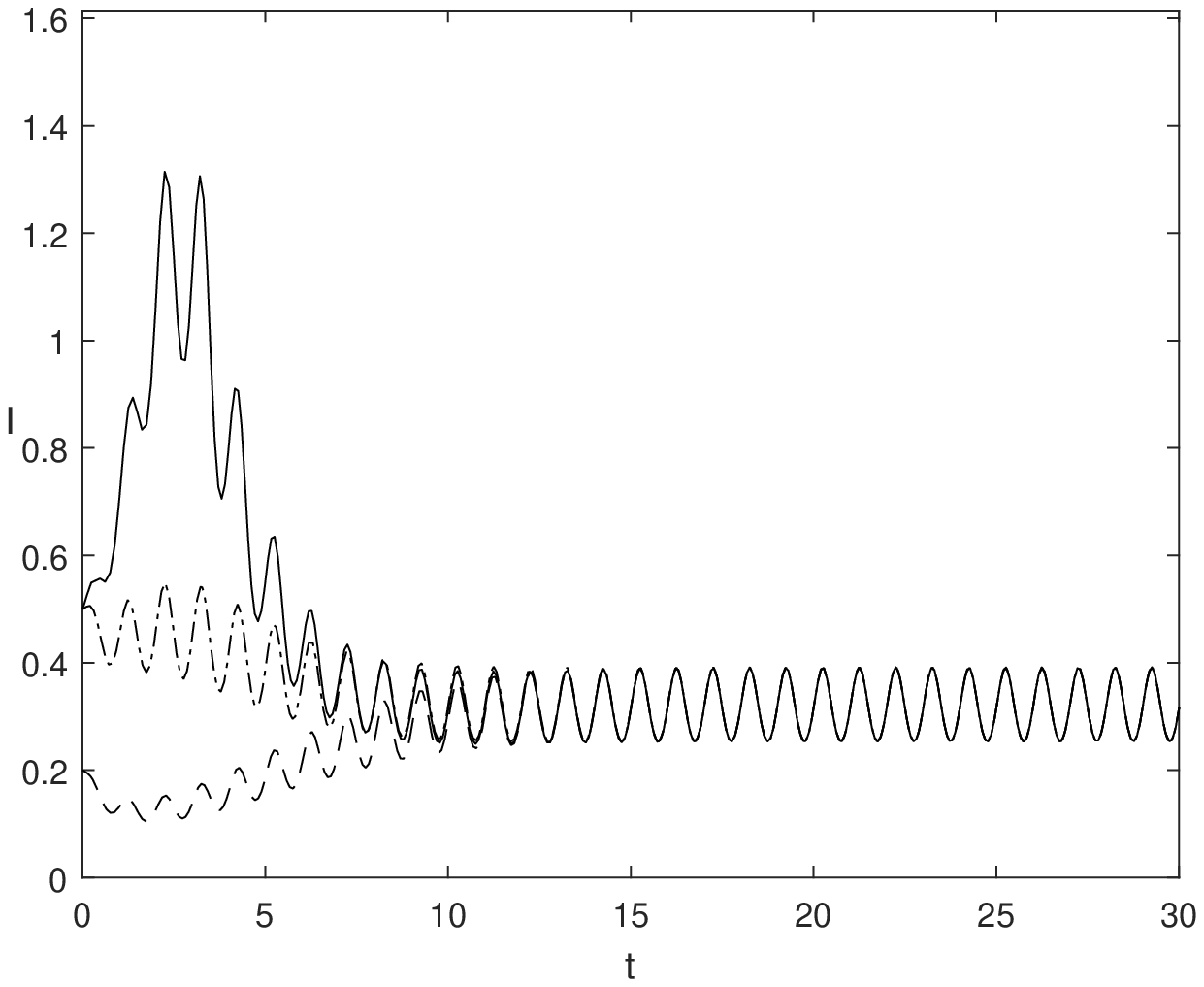}
  \end{minipage}
  \begin{minipage}[b]{.32\linewidth}
        \includegraphics[width=\linewidth]{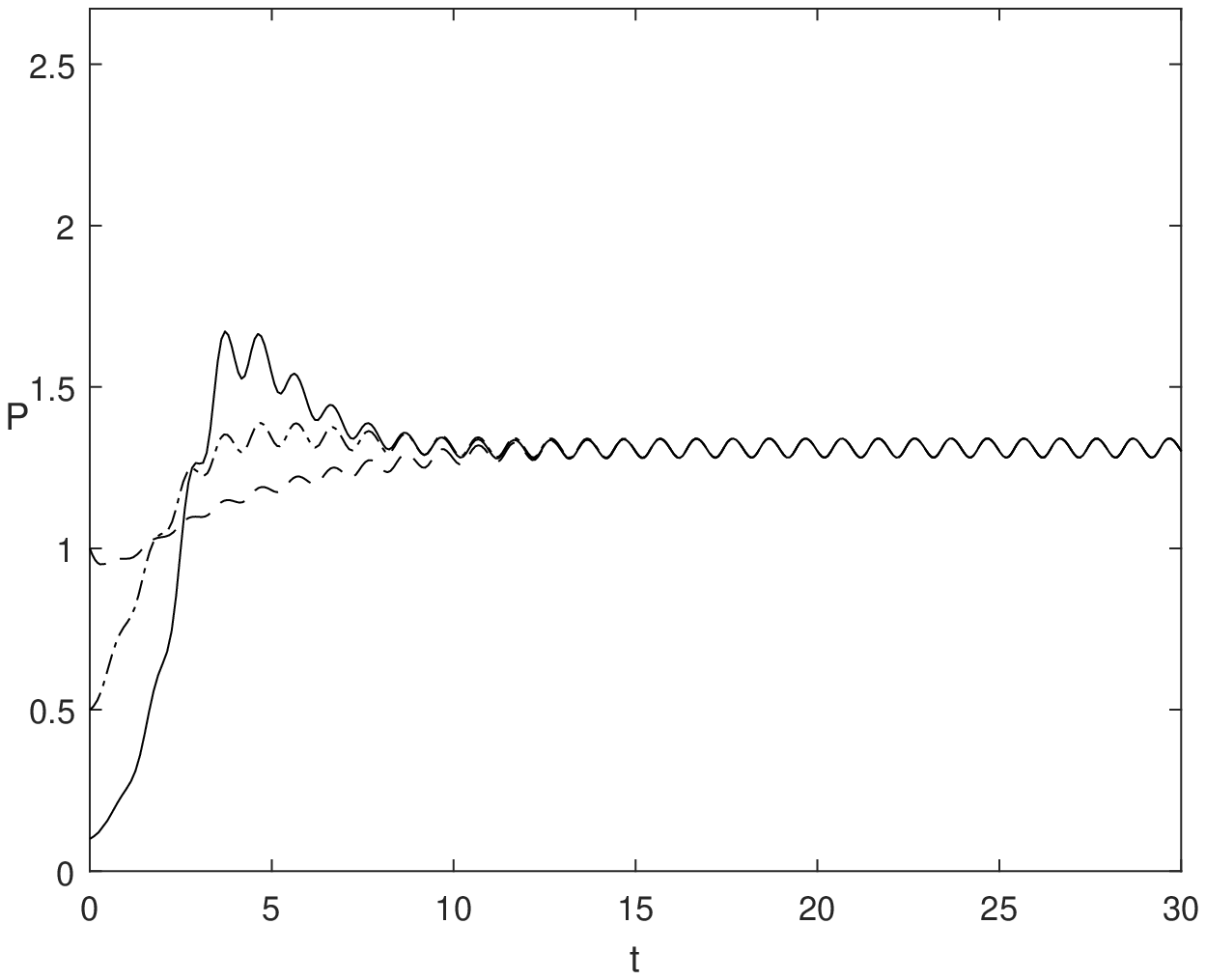}
  \end{minipage}
    \caption{Uniform strong persistence: $\beta_0=0.3$.}
      \label{fig_exe4_persistencia}
\end{figure}

\bibliographystyle{elsart-num-sort}

\begin{thebibliography}{10}
\expandafter\ifx\csname url\endcsname\relax
  \def\url#1{\texttt{#1}}\fi
\expandafter\ifx\csname urlprefix\endcsname\relax\def\urlprefix{URL }\fi



\bibitem{Dobson-QRV-1988} A. P. Dobson, \emph{The population biology of parasite-induced changes in host behavior}, Q. Rev. Biol. 30 (1988), 139--165.

\bibitem{Friend-H-2002} M. Friend, Avian disease at the Salton Sea, Hydrobiologia, 161 (2002), 293-306.

\bibitem{Garrione-Rebelo-NARWA-2016} M. Garrione and C. Rebelo, Persistence in seasonally varying predator-prey systems via the basic reproduction, Nonlinear Anal. Real World Appl. 30 (2016), 73-98.

\bibitem{Goh-AN-1976} B. S. Goh, Global stability in two species interactions, J. Math. Biol. 3 (1976), 313-318.

\bibitem{Hethcote-Wang-Han-Ma-TPB-2004}  H. W. Hethcote, W. Wang, L. Han and Z. Ma, \emph{A predator-prey model with infected prey},
Theor. Popul. Biol. 66 (2004), 259-268.

\bibitem{Hsu-Hwang-Kuang-MB-2001} Hsu, S. B., Hwang, T. W. and Kuang, Y., Global analysis of the Michaelis--Menten-type ratio-dependent predator-prey system, J. Math. Biol. 42 (2001), 489-506.


\bibitem{Koopmans-Wilbrink-Conyn-Natrop-Nat-Vennema-Lancet-2004} M. Koopmans, B.Wilbrink, M. Conyn, G. Natrop, H. van der Nat and H. Vennema, \emph{Transmission of H7N7 avian influenza A virus to human beings during a large outbreak in commercial poultry farms in the Netherlands}, Lancet. 363 (2004), 587--593.

\bibitem{Krebs-Blackwell-Scientific-Publishers-1978} J. R. Krebs, Optimal foraging: decision rules for predators, In: Krebs, J. R., Davies, N.B.
(Eds.), Behavioural Ecology: an Evolutionary approach, First ed. Blackwell Scientific Publishers, Oxford, (1978), 23-63.


\bibitem{Lu-Wang-Liu-DCDS-B-2018} Yang Lu, Xia Wang, Shengqiang Liu, A non-autonomous predator-prey model with infected prey, Discrete Contin. Dyn. Syst. B 23 (2018), 3817-3836.

\bibitem{Niu-Zhang-Teng-AMM-2011} Xingge Niu, Tailei Zhang, Zhidong Teng, The asymptotic behavior of a nonautonomous eco-epidemic model with disease in the prey, Appl. Math. Model. 35 (2011), 457-470.

\bibitem{Rebelo-Margheri-Bacaer-JMB-2012} Rebelo, C., Margheri, A, Baca{\"e}r, N, Persistence in seasonally forced epidemiological models, J. Math. Biol. 64~(6) (2012), 933--949.



\bibitem{Silva-JMAA-2017} C. M.~Silva, Existence of Periodic Solutions for Eco-Epidemic Model with Disease in the Prey, J. Math. Anal. Appl. 53 (2017), 383-397.



\bibitem{Wang-Zhao-JDDE-2008} W.~Wang, X.-Q. Zhao, Threshold dynamics for compartmental epidemic models in periodic environments, J. Dynam. Differential Equations 20~(3) (2008), 699-717.



\end{thebibliography}

\end{document}